\documentclass[english,12pt]{amsart}
\usepackage{amssymb}
\usepackage{amsfonts}
\usepackage{amscd}
\usepackage{color}
\usepackage{tikz}

\providecommand{\noopsort}[1]{}

\definecolor{immi}{rgb}{0,.6,.1}

\newbox\removebox
\newcommand\remove[2][blue]{%
\setbox\removebox=\ifmmode\hbox{$#2$}\else\hbox{#2}\fi%
\leavevmode
\rlap{\textcolor{#1}{\vrule height0.8ex depth-0.5ex width\wd\removebox}}%
\box\removebox
}
\long\def\bigremove#1{%
\par\setbox\removebox=\vbox{#1}%
\vbox{%
\vbox to0pt{\hbox{\tikz\draw[color=blue,thick] (0,0) -- (\wd\removebox,-\ht\removebox)  (\wd\removebox,0) -- (0,-\ht\removebox);}}
\box\removebox
}
}

\usepackage{mathrsfs} 

\newcommand{\cCexp}{\cC^{\mathrm{exp}}}
\newcommand{\cCexps}{\cC^{\mathrm{exp}}_s}
\newcommand{\cCs}{\cC_s}

\newcommand{\cQexp}{\cQ^{\mathrm{exp}}}

\newcommand{\Vol}{\operatorname{Vol}}

\newcommand{\Locp}{{\mathrm{Loc}^{0}}}

\newcommand{\Int}{\operatorname{Int}}
\newcommand{\Bdd}{\operatorname{Bdd}}
\newcommand{\Iva}{\operatorname{Iva}}

\newcommand{\id}{\operatorname{id}}

\def\VF{\mathrm{VF}}
\def\Res{\RF}
\def\VG{\mathrm{VG}}
\def\VGinf{\VG_{\infty}}
\newcommand{\RF}{{\rm RF}}

\def\ac{{
{\rm ac}}}
\def\cross{{\overline{\rm cross}}}

\def\gLPas{\cL_{\rm gDP}}

\def\gTPas
{{\rm gDP}}

\def\LPres{\cL_{\rm Pres}}
\def\TPres
{{\rm PRES}}
\def\Tdoag{\rm DOAG}

\def\res{\operatorname{res}}
\def\rv{\operatorname{rv}}
\def\RV{\operatorname{RV}}

\def\11{{\mathbf 1}}
\def\AA{{\mathbb A}}

\def\CC{{\mathbb C}}

\def\FF{{\mathbb F}}

\def\LL{{\mathbb L}}

\def\NN{{\mathbb N}}

\def\QQ{{\mathbb Q}}
\def\RR{{\mathbb R}}

\def\ZZ{{\mathbb Z}}

\def\cC{{\mathscr C}}
\def\cD{{\mathcal D}}

\def\cF{{\mathcal F}}

\def\cK{{\mathcal K}}
\def\cL{{\mathcal L}}
\def\cM{{\mathcal M}}

\def\cO{{\mathcal O}}

\def\cQ{{\mathcal Q}}

\def\cT{{\mathcal T}}

\def\llb{\mathopen{[\![}}

\def\rrb{\mathopen{]\!]}}

\newtheorem{thm}[subsubsection]{Theorem}
\newtheorem{lem}[subsubsection]{Lemma}
\newtheorem{cor}[subsubsection]{Corollary}
\newtheorem{prop}[subsubsection]{Proposition}

\theoremstyle{definition}
\newtheorem{defn}[subsubsection]{Definition}

\newtheorem{def-prop}[subsubsection]{Proposition-Definition}
\newtheorem{def-theorem}[subsubsection]{Theorem-Definition}
\newtheorem{def-lem}[subsubsection]{Lemma-Definition}

\theoremstyle{remark}
\newtheorem{remark}[subsubsection]{Remark}

\theoremstyle{plain}

\numberwithin{equation}{subsection}

  {\par\medskip\noindent #1\par\begingroup%
    \advance\leftskip by 1em\advance\rightskip by 1em}%
  {\par\endgroup}

\DeclareMathOperator*{\sq}{\square}
\newcommand{\ord}{\operatorname{ord}}
\newcommand{\Jac}{\operatorname{Jac}}

\usepackage[T1]{fontenc}
\usepackage{mathrsfs}

\definecolor{immi's color}{rgb}{0,.6,.1}

\title[Integration, uniform in all local fields of characteristic zero]{Integration of functions of motivic exponential class, uniform in all non-archimedean local fields of characteristic zero}
\author[R.~Cluckers]{Raf Cluckers}
\address{Universit\'e de Lille,
Laboratoire Painlev\'e,
 CNRS - UMR 8524, Cit\'e Scientifique, 59655
Villeneuve d'Ascq Cedex, France, and,
KU Leuven, Department of Mathematics,
Celestijnenlaan 200B, B-3001 Leu\-ven, Bel\-gium\\}

\email{Raf.Cluckers@math.univ-lille1.fr}
\urladdr{http://rcluckers.perso.math.cnrs.fr/}

\author[I.~Halupczok]{Immanuel Halupczok}
\address{Lehrstuhl f\"ur Algebra und Zahlentheorie, Mathematisches Institut, Universit\"atsstr. 1, 40225 D\"usseldorf, Germany}
\email{math@karimmi.de}
\urladdr{http://www.immi.karimmi.de/en/}



\thanks{The authors would like to thank Emmanuel Kowalski and Fran\c cois Loeser for interesting discussions, and thank E.~Kowalski the Forschungsinstitut f\"ur Mathematik (FIM) at ETH Z\"urich for the hospitality and invitation to R.C.~to give the Nachdiplom Lectures at the ETH in 2014 related to the themes of the paper. The authors were supported by the European Research Council under the European Community's Seventh Framework Programme (FP7/2007-2013) with ERC Grant Agreement nr. 615722 
MOTMELSUM, by the Labex CEMPI  (ANR-11-LABX-0007-01); the author I.H.~was supported by the SFB~878 of the Deutsche Forschungsgemeinschaft.}

\begin{document}

\begin{abstract}
Through a cascade of generalizations, 
we develop a theory of motivic integration which works uniformly in all non-archimedean local fields of characteristic zero,
overcoming some of the difficulties related to ramification and small residue field characteristics.
We define a class of functions, called functions of motivic exponential class,
which we show to be stable under integration and under Fourier transformation,
extending results and definitions from \cite{CLexp}, \cite{CLbounded} and \cite{CGH}. We prove uniform results related to rationality and to various kinds of loci. A key ingredient is a refined form of Denef-Pas quantifier elimination which allows us to understand definable sets in the value group and in the valued field.
\end{abstract}

\maketitle

\section{Introduction}

\subsection{}


Much of the existing theory of local zeta functions and $p$-adic integrals has been developed for large residue field characteristic, and in the case of small residue field characteristic only with bounds on ramification. Sometimes these restrictions come from resolution of singularities with good reduction modulo (large) $p$ (see e.g.~\cite{DL}, \cite{Denefdegree} and Theorems (3.3) and (3.4) of \cite{DenefBour}), and sometimes they come from quantifier elimination and model theoretic results (see e.g.~\cite{Pas2}, \cite{CDo}, \cite{CLbounded}, \cite{CGH}, \cite{HK}). Sometimes however, arbitrary ramification and even positive characteristic local fields can be allowed, for example in situations with some smoothness or smooth models, see e.g.~\cite{LoeserSeb}, \cite{Sebag2004}, \cite{NicSeb}, \cite{NicaSeba}, and in situations where variants of Hironaka's resolution can be used over $\QQ$ like for Theorem E of \cite{AizDr} about wave front sets and for the rationality result from the 1970s by Igusa, see Theorem 8.2.1 of \cite{Igusa:intro} or Theorem (1.3.2) of \cite{DenefBour}.

\subsection{}
In this paper we remove some of the restrictions on the model theoretic approach by refining quantifier elimination results, and grasp the rewards
to the construction of a framework of integration which works uniformly in all non-archimedean local fields of characteristic zero, extending recent work from \cite{CLexp},
\cite{CLbounded} and \cite{CGH}. By a non-archimedean local field, local field for short, we mean a finite field extension of $\QQ_p$ for some prime $p$ or $\FF_q((t))$ for some prime power $q$.
For $K$ a local field with valuation ring $\cO_K$ with maximal ideal $\cM_K$, we do not obtain new results about $\cO_K$ modulo the ideals $n\cM_K := \{ nm\mid m\in \cM_K\}$ for integers $n>0$,
but rather, we use these finite quotients as tools (one might even say `oracles'), in order to understand the model theory of $K$ and the geometry of definable sets.
Let us note that the use of model theory to study $p$-adic integrals originated in work by Denef \cite{Denef} (enabled by Macintyre's quantifier elimination result \cite{Mac}), where the approach with resolution of singularities was used by Igusa in the early seventies (enabled by Hironaka's result \cite{Hir:Res}).

\subsection{}
The new framework thus removes the bounds on ramification degrees from \cite{CLbounded}, is stable under Fourier transformation as in \cite{CLexp}, and deals with uniformity in local fields of characteristic zero.
This yields several kinds of new uniformities for the behaviour of $p$-adic integrals and for bad (or exceptional) loci.
In the afore-mentioned Theorem 8.2.1 of \cite{Igusa:intro}, it is the set of candidate poles and the form of the denominator that is completely uniform over all local fields of characteristic zero; in Theorem E of \cite{AizDr} it is the wave front which is included in a Zariski closed set of controlled dimension which is completely uniform over all local fields of characteristic zero. These two phenomena should now find a common ground in the uniform treatment of this paper, see Sections \ref{sec:loci} and \ref{loc:zeta}.
For the sake of simplicity, we do not take an abstract motivic approach.

\subsection{}
An important step for treating arbitrary ramification via model theory was provided by S.~Basarab \cite{Basarab} and its quantifier elimationation result which can be reformulated in several ways, e.g.~with the generalized Denef-Pas language.

Key for us is a refinement of the classical Denef-Pas and Basarab quantifier elimination results: we eliminate both valued field and value group quantifiers, regardless of ramification, see Theorem \ref{QEZ}. This leads to a more subtle situation than in the cases with bounded ramification, and only a weak form of orthogonality survives.  A cascade of generalizations of results related to the geometry of definable sets and integration follows uniformly in all local fields of characteristic zero.



\subsection{}
Let us describe some examples of uniform behaviour. 
Recall that definable functions are field-independent descriptions of functions which generalize in particular polynomial mappings; see Section~\ref{s:gDP} for precise definitions.

\par


Let $n>0$ be an integer, and $f$ be a definable function from the $n$-th Cartesian power of the valuation ring and taking values in the value group. In particular, for any $p$-adic field $K$ (namely, any finite field extension of $\QQ_p$ for any prime $p$), $f$ yields a function $f_K:\cO_K^n\to \ZZ$. Since Denef's results in \cite{Denef}
one knows, under natural integrability conditions, and if one puts for real $s>0$  
\begin{equation}\label{zeta}
Z_K(s) := \int_{x\in \cO_K^n} q_K^{-s f_K(x)}  |dx|,
\end{equation}
that $Z_K(s)$ is rational in $q_K^{-s}$ where $q_K$ is the number of residue field elements. Moreover, Denef \cite{Denef}
showed that the denominator always divides a polynomial of a simple form, namely a finite product of factors of the form $q_K^{bs}$ and
$$
1-q_K^{a_i+b_is}
$$
for some integers $a_i,b$ and $b_i\not=0$, depending on $K$. The dependence on $K$ under higher and higher ramification remained highly unstudied. By the uniform treatment of this paper, we find that the list of candidate poles is finite, even when $K$ varies over all local fields of characteristic zero. More precisely, there are $b\in\ZZ$, nonzero $c\in \QQ$, and a finite collection of pairs of integers $(a_i,b_i)$ with $b_i\not=0$, for $i=1,\ldots,N$ for some $N$, such that for any $p$-adic field $K$,
$$
Z_K(s) q_K^{(b + \ord c ) s}\prod_{i=1}^N(1-q_K^{a_i+b_is})
$$
is a polynomial in $q_K^{-s}$. See Section \ref{loc:zeta} for more general rationality results.

\subsection{}
More generally, we extend the framework of constructible exponential functions from \cite{CLexp} to all local fields of characteristic zero.
(We will call them functions ``of motivic exponential class'', or ``of $\cCexp$-class'', for short.)
Stability under integration of functions of $\cCexp$-class implies the above finiteness of candidate poles. In the case that $f$ (in (\ref{zeta})) is the order of a polynomial over $\QQ$,
this application was already known to Igusa in the 1970's by (embedded) resolution of singularities over $\QQ$, see Theorem 8.2.1 of \cite{Igusa:intro}.
Also, given a definable $f$, the application was shown by Pas for large enough residue field characteristic \cite{Pas}, and for small residue field characteristic but with bounded ramification \cite{Pas2}. Both cases treated by Pas rely on Denef-Pas quantifier elimination (the model theoretic approach).

\par

\subsection{}
Our formalism can be used to study loci. First, we deduce that certain bad or exceptional loci are small; see for example Theorem \ref{badlocus}.
Roughly, the idea is that loci of several kinds of bad 
behaviour are contained in proper Zariski closed subsets, uniformly in all local fields of characteristic zero, roughly as in Theorem E of \cite{AizDr}.
Many such results are already known for large enough residue field characteristic (or assuming bounds on the ramification), so that the new point is again to be completely uniform in all local fields of characteristic zero. Secondly, the study of various kinds of loci and of extrapolations is generalized from \cite{CGH} to our setting in Section \ref{sec:loci}.

\subsection{}

In a certain sense, this paper covers a big part of the material of the course given as Nachdiplom Lectures at the ETH of Z\"urich in 2014 by the first author, where the feature to deal with all $p$-adic fields was introduced. We chose to give a didactical presentation of the results and to give complete proofs of all results in the present generality. This complements the related work of \cite{CLexp}, \cite{CLbounded}, \cite{CGH} by generalizing but also by developing almost all proofs in a single paper. Furthermore, we develop a naturality result for our classes of functions in Section \ref{sec:natur}.
The main technical novelties are related to quantifier elimination and a weak form of orthogonality; they are treated at the end of this paper, in Section \ref{sec:qe}.

Note that our framework bears nothing new in the positive characteristic case: in the small positive characteristic case deep mysteries remain, and
the large positive characteristic case can be treated on a similar footage as the large residue field case in mixed characteristic and is already developed in \cite{CLexp}, \cite{CGH}. A continued analysis of $\cCexp$-functions is developed in \cite{CGH5}.




\section{Uniform $p$-adic definable sets and functions}\label{s:gDP}

We introduce a language which we use to fix our notion of $p$-adic definable sets in a  uniform way across all finite field extensions of $\QQ_p$ for all primes $p$.
Our language has angular components and allows us to eliminate both valued field and, importantly, value group quantifiers.
This helps to control the geometrical difficulties, as do (weak) orthogonality, cell decomposition, and the Jacobian Property, see Section \ref{sec:qe}.
With these definable sets and functions and these results, we are able to build up the class of functions that are stable under integration and Fourier transformation, uniformly over all local fields of characteristic zero.

First we give some general definitions about valued fields and a generalization of the Denef-Pas language.

\subsection{Residue rings and angular component maps}\label{genval}

For $L$ a valued field\footnote{A field $L$ together with a surjective map $\ord: L \to \VG_L\cup\{+\infty\}$, with $\ord(0)=+\infty$ and with $\VG_L$ an ordered abelian (additively written) group, is called a valued field if $\ord(x+y)\geq \min(\ord x, \ord y)$ for all $x,y$ in $L$ and if moreover $\ord$ restricts to a group homomorphism $L^\times\to \VG_L$. By an ordered abelian group we mean an abelian group with a total order and such that $a<b$ implies $a+c< b+c$ for all group elements $a,b,c$.}
 with valuation map $\ord$, write $\cO_L:=\{x\in L\mid \ord (x)\geq 0 \}$ for the valuation ring with maximal ideal $\cM_L:=\{x\in L\mid \ord (x)> 0 \}$, residue field $\RF_L = \cO_L/\cM_L$, and additively written value group $\VG_L$. Write $\VGinf{}_{L}$ for the disjoint union $\VG_L\cup \{+\infty\}$.
For any integer $n>0$, write
$$
\RF_{n,L}
$$
for the quotient $\cO_L/(n\cM_L)$.
Write
$$
\res_n:L\to \RF_{n,L}
$$
for the projection $\cO_L\to \RF_{n,L}$ extended by zero outside $\cO_L$, and
$$
\res_{m,n}: \RF_{m,L}\to \RF_{n,L}
$$
for the projection map, for positive integers $n$ dividing $m$.

A collection of maps
$$\ac_n:L\to \RF_{n,L}$$
for integers $n>0$ is called a compatible system of angular component maps if for each $n$, $\ac_n$ is a multiplicative map from $L^\times$ to $\RF_{n,L}^\times$, extended by zero on zero, such that moreover $\ac_n$ coincides with $\res_n$ on $\cO_L^\times$, and, for $n$ dividing $m$, the maps $\ac_n$, $\ac_m$, and $\res_{m,n}$ form a commutative diagram.

\begin{remark}
It is important to note that $(n\cM_L)$ is the ideal of all $nm$ with $m\in \cM_L$, and
(usually) not the $n$-th power of the maximal ideal. The residue ring $\RF_{n,L}$ is 
different from the residue field of $L$ if and only if the characteristic of $\RF_{L}$ divides $n$.
\end{remark}

\subsection{The generalized Denef-Pas language}\label{ss:gDP}

Consider the many sorted first order language $\gLPas$ with sorts $\VF$, $\Res_n$ for each integer $n>0$, and $\VGinf$, and with the following symbols.
On $\VF$ and on each of the $\Res_n$ one has a disjoint copy of the ring language having symbols
$$
+,-,\cdot,0,1.
$$
On $\VGinf$  one has the language $\cL_{\rm oag,\infty}$, namely the constant symbol $+\infty$ together with the language $\cL_{\rm oag}$ of ordered abelian groups, with symbols
$$
+,-,0,<. 
$$
Furthermore one has 
the following function symbols for all positive integers $n$:
\begin{itemize}

\item

 $\ord:\VF\to \VGinf$

\item

 $\ac_n:\VF\to\Res_n$.



\end{itemize}

Let us call the language $\gLPas$ the generalized Denef-Pas language.

The generalized Denef-Pas language is designed to study (definable sets in) henselian valued fields $L$ of characteristic zero regardless of ramification degrees.

A definitional expansion of $\gLPas$ 
yields, regardless of ramification, quantifier elimination in the valued field, and, under some extra conditions, also in the value group, see Theorem \ref{QEZ}.

\subsection{Generalized Denef-Pas structures}\label{ss:gendp}

A generalized Denef-Pas structure on a valued field 
$L$ as in Section \ref{genval} consists of interpretations of all the sorts and the symbols of $\gLPas$, subject to the following natural conditions.

\begin{itemize}

\item

The sorts $\VF$, resp.~$\Res_n$ and $\VGinf$ have as interpretations $L$, resp.~$\Res_{n,L}$, both with the ring structure, and  $\VGinf{}_{L} = \VG_L \cup \{+\infty\}$ with the structure of an ordered abelian group on $\VG_L$,
 and the natural meaning for $+\infty$.

\item

The map $\ord$ is the valuation map as in Section~\ref{genval}.

\item

The maps $\ac_n:L\to \Res_{n,L}$ form a compatible system of angular component maps.

\end{itemize}

We define the $\gLPas$-theory $\gTPas$ to be the theory of the generalized Denef-Pas structures on valued fields $L$ such that moreover $L$ is a henselian valued field of characteristic $0$ (and arbitrary residue field characteristic).

\subsection{$p$-adic fields as generalized Denef-Pas structures}\label{p-adic}

For now and until the end of Section~\ref{sec:i}, the only generalized Denef-Pas structures we are interested in are $p$-adic fields, i.e., finite field extensions of $\QQ_p$ for some prime number $p$.

Let us write $\Locp$ for the collection of all local fields of characteristic zero, equipped with a uniformizer\footnote{A uniformizer for $\cO_K$ is any element in $\cO_K$ with minimal positive valuation.}
$\varpi_K$ for $\cO_K$. Such a uniformizer induces a compatible system of angular component maps: the map
$\ac_n:K\to \RF_{n,K}$ sends $0$ to $0$ and any nonzero $x$ to $x\varpi_K^{-\ord x}\bmod (n\cM_K)$.
In this way, we consider fields $K$ in $\Locp$ as
generalized Denef-Pas structures.
Note that any compatible system $\ac_n$ on a $p$-adic field arises in this way from a uniformizer $\varpi_K$ and that vice versa, the maps $\ac_n$ determine $\varpi_K$.

For fields $K \in \Locp$, we use the following notations and conventions:
Write $q_K$ for the number of elements in the residue field $\RF_K$ of $K$, and $p_K$ for its characteristic. We identify the value group of $K$ with $\ZZ$, so that $\varpi_K$ has valuation $1$.

\subsection{Uniform $p$-adic definable sets}
\label{unifdef}
We now introduce the notion of definable sets adapted to the class $\Locp$ of fields we are interested in,
i.e, sets which are $\gLPas$-definable uniformly in local fields of characteristic zero, more precisely, uniformly in $K \in \Locp$. Since this is the general framework until the end of Section~\ref{sec:i}, we will simply call them ``definable sets''.

A definable set
$$
X=(X_K)_{K\in\Locp}
$$
is a collection of sets such that there is an $\gLPas$-formula $\varphi$
such that
$$
X_K=\varphi(K),
$$
where $\varphi(K)$ is the definable subset of a Cartesian power of the universes $K$, $\RF_{n,K}$, and $\ZZ$ defined by $\varphi$ in the sense of model theory.\footnote{Note that, for a definable set $X=(X_K)_{K}$, each $X_K$ is in fact a subset of some `affine' coordinate space, as is standard in model theory.}

By abuse of notation, we will use the notation for the sorts $\VF, \dots$ also for the corresponding definable sets: We write
$$
\VF^n\times \prod_{i=1}^N \Res_{m_i}\times \VG^r
$$
for the definable set $(X_K)_K$ with $X_K = K^n\times \prod_{i=1}^N \RF_{m_i,K}\times \ZZ^r$.

A collection
$$
(f_K:X_K\to Y_K)_{K\in\Locp}
$$
of functions for definable sets $X=(X_K)_K$ and $Y=(Y_K)_K$ is called a definable function if the collection of the graphs is a definable set. We also write
$$
f:X\to Y
$$
for $(f_K:X_K\to Y_K)_{K \in\Locp}$.

\section{Functions of $\cC$-class and of $\cCexp$-class}\label{sec:mot}

In this section we introduce ``functions of motivic exponential class'' ($\cCexp$-functions, for short);
these are functions which are given uniformly in all local fields of characteristic zero,
more precisely, uniform in all $K\in \Locp$. On our way, we first introduce a smaller class, called $\cC$-functions.
Compared to the notions $\cC$ and $\cCexp$ of \cite{CLoes} \cite{CLexp}, the present context may be considered as `semi-motivic', since, for the sake of simplicity, we do not allow other valued fields than local ones.

\subsection{$\cC$-functions}\label{sec:cC-func}

Let $\AA$ be the
ring of the following rational functions over $\ZZ$
$$
\AA := \ZZ[q , 1/q , \bigcup_{i>0}\{\frac{1}{1-q^{-i}} \} ] 
$$
where $q$ is a formal variable and $i$ runs over positive integers.

Note that any element $a(q)\in \AA$ can be evaluated at any real number $q=q_0$ with $q_0>1$.

For a definable set $X$, by a function $f\colon X \to \CC$ we mean a tuple $f = (f_K\colon X_K \to \CC)_{K\in\Locp}$. We turn the set of functions $f\colon X \to \CC$ into a ring using pointwise addition and multiplication,
namely,
$$
f_1+f_2 = (f_{1K}+f_{2K} : X_K\to \CC)_K
$$
for $f_1, f_2\colon X_K \to \CC$ and, likewise,
$$
f_1\cdot f_2 = (f_{1K}\cdot f_{2K}:X_K\to \CC)_K.
$$

For a definable set $X$, the ring of $\cC$-functions on $X$ (or ``functions of $\cC$-class on $X$'') is denoted by $\cC(X)$ and is defined as the subring of the real-valued functions $X\to\RR$ generated by the following elements (with pointwise operations):

\begin{enumerate}
\item\label{cC1} $a : X\to \RR$ for any $a\in \AA$, where $a_K(x) := a(q_K)$ for $x\in X_K$ and  $K\in \Locp$.

\item\label{cC2} $\alpha : X\to \RR$ for any $\VG$-valued definable function $\alpha$, with the obvious meaning of $\alpha_K(x)$ for $x\in X_K$ and  $K\in \Locp$.

\item\label{cC3} $q^{\beta}: X\to \RR$ for any $\VG$-valued definable function $\beta$, with $(q^\beta)_K(x) := q_K^{\beta_K(x)}$ for $x\in X_K$ and  $K\in \Locp$.

\item\label{cC4} $\# Y : X\to \RR$ for any definable subset $Y$ of $X\times \prod_{t=1}^{\ell} \Res_{n_{t}}$ for some $\ell\geq 0$ and some $n_{t} > 0$, and where $(\# Y)_K(x) := \# (  Y_{K,x}  )$ for $x\in X_K$ and  $K\in \Locp$, with $Y_{K,x}$ consisting of $y$ such that $(x,y)$ lies in $Y_K$.
\end{enumerate}

In other words, a collection of functions
$$
f=(f_K:X_K\to\RR)_{K\in\Locp}
$$
is of $\cC$-class if and only if there are $a_i\in \AA$, integers $N\geq 0$, $s_i\geq 0$, $\ell_i\geq 0$, $n_{i,t}\> 0$, $\VG$-valued definable functions $\beta_i$ and $\alpha_{i,j}$ on $X$, and definable subsets $Y_i$ of $X\times \prod_{t=1}^{\ell_i} \Res_{n_{i,t}}$, such that
\begin{equation}
f_K(x) = \sum_{i=1}^N a_i(q_K) (\# Y_{i,K,x}) q_K^{ \beta_{i,K}(x) } \prod_{j=1}^{s_i} \alpha_{i,j,K}(x)
\end{equation}
for all $K\in \Locp$ and all $x\in X_K$. Indeed, products of generators of the form other than (\ref{cC2}) can be combined.

For fixed $K\in \Locp$, these functions in $\cC$ were first studied (in the context of stability under integration as in Theorem \ref{p1qmotexp} below) in  \cite{Denef1}, in a motivic way in \cite{CLoes}, and, in a uniform $p$-adic way, but for large $p$, in \cite{CLexp}. The notation of $\cC(X)$ resembles the one of \cite{CLoes}, the difference being that $X$ here is a definable set, and in \cite{CLoes} it is a definable subassignment.

\subsection{$\cCexp$-functions}

For any $p$-adic field $K \in \Locp$, write $\cD_K$ for the collection of additive characters $\psi:K\to\CC^\times$
which are trivial on $\cM_K$ but non-trivial on $\cO_K$; note that this definition is slightly different from previous
papers.\footnote{Previous papers only considered $\psi$ which induce a fixed map on the residue field. Which definition to use is
mostly a matter of taste; to our taste, specifying the fixed map on the residue field is a bit technical.}

Notationally, we treat $\cD = (\cD_K)_{K \in \Locp}$ in a similar way as definable sets:
$\cD \times X$ stands for $(\cD_K \times X_K)_{K \in \Locp}$,
and by a function $f:\cD \times X \to \CC$, we mean a family of functions $(f_K: \cD_K \times X_K \to \CC)_{K \in \Locp}$.

For a definable set $X$, the ring of $\cCexp$-functions on $X$ (also called ``functions of $\cCexp$-class'') is denoted by $\cCexp(X)$ and is defined as the ring of complex-valued functions on $\cD \times X$ consisting of finite sums of functions sending $(\psi,x,K)$ with $\psi\in\cD_K$, $x\in X_K$, and $K\in \Locp$ to
\begin{equation}\label{eq:gen5}
f_K(x)\cdot \sum_{y \in Y_{K,x} } \psi\big( h_K(x ,y ) + \frac{e_K(x ,y )}{N}\big),
\end{equation}
for $f\in \cC(X)$ and definable functions
$h:Y\to \VF$ and $e:Y\to \Res_N$ for some integer $N>0$, where $Y$ is a definable subset of
$X\times \prod_{i=1}^\ell \Res_{n_i}$ for some $\ell\geq 0$ and some $n_i > 0$.

Here, by $\psi( h + e/N)$ for some $h\in K$ and some $e\in \RF_{N,K}$ we mean $\psi( h + e'/N)$ for any $e'\in \cO_K$ with $\res_N(e') = e $, which is independent from the choice of $e'$ since $\psi$ lies in $\cD_K$.

Clearly $\cC(X)$ can be considered as a subring of $\cCexp(X)$.

For a function $g$ in $\cCexp(X)$ for $K\in\Locp$ and $\psi \in \cD_K$, we write $g_{K,\psi}:X_K\to \CC$ for the restriction of $g$ to $\{\psi\}\times X_K$.

In other words, a collection of functions
$$
g=(g_{K,\psi}:X_K\to\CC)_{K\in\Locp,\psi\in \cD_K}
$$
is of $\cCexp$-class if and only if there are $f_i\in \cC(X)$, $M\geq 0$, $N_i\geq 0$, $\ell_i\geq 0$, $n_{i,t} > 0$, definable subsets $Y_i$ of $X\times \prod_{t=1}^{\ell_i} \Res_{n_{i,t}}$ and definable functions $h_{i}:Y_i\to\VF$ and $e_{i}:Y_i\to \Res_{N_i}$, such that
\begin{equation}\label{eq:expSum}
g_{K,\psi}(x) = \sum_{i=1}^M f_{i,K}(x) \sum_{y\in Y_{i,K,x}} \psi\big( h_{i,K}(x ,y ) + \frac{e_{i,K}(x ,y )}{N_i}\big).
\end{equation}

\section{Integration of $\cC$- and $\cCexp$-functions}\label{sec:i}

The proofs in this section rely on the results from Section \ref{sec:qe}. To readers wishing to understand these proofs in detail, we recommend to come back to the proofs of this section after reading Section \ref{sec:qe}.


\subsection{Integration}\label{subsec:int}

The functions defined in the previous section have very good behaviour under integration, and, in particular, under Fourier transformation.
For any local field $K$ and any integer $n>0$, put the additive Haar measure on $K$ (normalized so that $\cO_K$ has measure $1$),
the counting measure on $\ZZ$ and on $\RF_{n,K}$, and the product measure on Cartesian products of such sets.

\begin{thm}[Stability under Integration]\label{p1qmotexp}
Let $f$ be in $\cCexp(W)$ for some definable sets $X$, $Y$ and $W\subset X\times Y$. Then there exists $g$ in $\cCexp(X)$ such that the following holds for all $K$ in $\Locp$, all $x\in X_K$, and all $\psi\in \cD_K$
$$
g_{K,\psi}(x) = \int_{y\in W_{K,x}}f_{K,\psi}(x,y),
$$
whenever the function $y\mapsto f_{K,\psi}(x,y)$ is integrable over $W_{K,x}$ against the product measure described just above the theorem. Moreover, if $f$ lies in $\cC(W)$, then $g$ can be taken in $\cC(X)$.
\end{thm}

For readers familiar with the proofs in big residue characteristic, note that there are two main differences in the present setting:
One is that $\RF$ is replaced by the sorts $\RF_n$ throughout. The second one is more tricky and concerns the fact that
$\VG$ and the sorts $\RF_n$ are not orthogonal anymore, and in particular,
definable subsets of $\VG$ need not be Presburger definable. This problem is solved using Corollary~\ref{linear}, which states
that definable sets in $\VG$ can be made Presburger definable at the cost of reparameterization, i.e., introducing new $\RF_{n}$-variables. Below is a complete proof of Theorem \ref{p1qmotexp}, implementing these modifications of the proof of Theorem 4.4.3 of \cite{CGH}.
First we recall a lemma already used in \cite{CLoes} about geometric power series, their derivatives, and their summation properties.

\begin{lem}[Lemma 4.4.3 of \cite{CLoes}]\label{trl} Let $R$ be a ring (commutative and with unit)
and let $P$ be a degree $d$ polynomial in $R [X]$. The equality
\begin{equation}\label{delta}
\sum_{n \geq a} P (n) T^n
=
\sum_{i = 0}^d \frac{[\Delta^i P (a)] T^{ a + i}}{(1 - T)^{i + 1}}
\end{equation}
holds in $R\llb T \rrb$ for all $a$ in $ \NN$. Here $\Delta^i$ is the
$i$-th iterate of the difference operator $P \mapsto P (X + 1) - P
(X)$ with the convention $\Delta^0 P = P$. \qed
\end{lem}
\begin{proof}
This follows by induction from
$(1 - T)\cdot\sum_{n \geq a} P (n) T^n = P(a)\cdot T^a + T \cdot \sum_{n \geq a} \Delta P (n) T^n$ (and using $\Delta^{d+1} P = 0$).
\end{proof}

\begin{proof}[Proof of Theorem~\ref{p1qmotexp}]
By Fubini's Theorem, it is enough to treat the case that $Y$ is either $\VF$, $\VG$ or $\RF_{n}$ for some integer $n>0$.

The case that $Y$ is $\RF_{n}$ follows from the definitions of functions of $\cCexp$-class and of $\cC$-class, where sums over $\RF_{n,K}$ are built in.

Let us now treat the case that $Y$ is $\VG$, using terminology and results from Section \ref{sec:qe}. Let us first suppose that the definable functions which take values in $\VG$ and which appear in the build-up of $f$ (namely in the forms of generators (\ref{cC2}) and (\ref{cC3}) of  Section \ref{sec:cC-func}) are linear over $X$, that $W$ is the definable set
\begin{equation}\label{eq:W}
\{(x,y)\in X\times \VG\mid \alpha(x) \leq y \sq \beta(x)\}
\end{equation}
where $\sq$ is $<$ or no condition and where $\alpha,\beta:X\to\VG$ are definable functions, and that all other build-up data of $f$ (namely, generators  (\ref{cC4}) of  Section \ref{sec:cC-func} and $h$ and $e$ as in (\ref{eq:gen5})) factor through the projection $W\to X$. Then the conclusion follows from Lemma \ref{trl}. Indeed, for any $K$ in $\Locp$, any $x\in X_K$ and any $\psi$ in $\cD_K$, $f_{K,\psi}(x,y)$ is a finite sum of terms $T_i$ of the form
\begin{equation}\label{eq:Ti}
c_{i,K,\psi}(x)y^{a_i}q_K^{b_iy}
\end{equation}
for integers $a_i\geq 0$, rational numbers $b_i$ and  $c_i$ in $\cCexp(X)$. The integrability of $f_{K,\psi}(x,y)$ over $y$ in $W_{K,x}$ is automatic when $\sq$ is $<$ and we get $g$ from Lemma \ref{trl}. When $\sq$ is no condition, we regroup the terms if necessary, so that the pairs $(a_i,b_i)$ are mutually different for different $i$. By observing different asymptotic behavior of these terms for growing $y$, we may consider the sub-sum $\sum_{i\in J}T_i$ with $i\in J$ if $b_i<0$. This time we apply Lemma \ref{trl} to this sub-sum to find $g$.

The Presburger results from Section \ref{sec:presburgering} (namely Corollary \ref{linear} and Proposition \ref{prop:recti}) together with the already treated case that $Y= \RF_{n}$ finish the general case that $Y = \VG$
(using that we already know how to treat $\RF_n$-variables).

Finally we treat the case that $Y$ is $\VF$.

By Theorem \ref{cd} and
by the already treated cases that $Y = \RF_n$ and $Y= \VG$, it is enough to treat the case that $W_{K,x}$ is a single open ball in $K$ for each $K$ and each $x\in X_K$.
Moreover, on these balls, we may assume that the functions like $h$ as in (\ref{eq:gen5}) that appear in the build-up of $f$ have the Jacobian property, and that all other build-up data of $f$ factors through the projection $W\to X$.
But then calculating the integral is easy, and goes as follows. By the Jacobian property for $h$, the set $h_K(x,W_{K,x})$ is an open ball, say, of valuative radius $n_K\in\ZZ$. If $n_K\leq 0$, then the integral of $\psi(h_K(x,y))$ over $y\in W_{K,x}$ equals zero for each $\psi$ in $\cD_K$. If $n_K>0$, then $\psi(h_K(x,y))$ is constant on $W_{K,x}$, say, equal to $\xi_{K,\psi,x}$, and hence the integral of $\psi(h_K(x,y))$ over
$y\in W_{K,x}$ equals the volume of the ball $W_{K,x}$, say $q_K^{m_K}$, times $\xi_{K,\psi,x}$. By Lemma \ref{avoidSkol} and by the already treated cases that $Y = \RF_n$, we may suppose that $h$ factors through the projection $W\to X$, so that $q_K^{m_K}\xi_{K,\psi,x}$ lies in $\cCexp(X)$ by definition of $\cCexp$-functions and we are done.
\end{proof}

Note that Theorem \ref{p1qmotexp} generalizes Theorem 4.4.3 of \cite{CGH} and Proposition 8.6.1 and Theorem 9.1.5 from \cite{CLexp} as explained in the introduction.

\subsection{Naturality of the $\cC$-class and $\cCexp$-class}\label{sec:natur}

\begin{prop}\label{prop:natcC}
Consider collections consisting of one ring $C(X)$ of functions
$$
f: X \to \RR
$$
(in the sense of Subsection~\ref{sec:cC-func}) for each definable set $X$. The collection $\cC(X)$ is the smallest such collection with the following
properties:
\begin{enumerate}
\item
For every definable set $X$, $C(X)$ contains the characteristic function of every definable set $A \subset X$.
\item
The collection is stable under integration,
namely, for any $X$ and $Y$ and any $f\in C(X\times Y)$ such that for each $K$ and each $x\in X_K$, the function $y\mapsto f_K(x,y)$ is integrable over $Y_K$, the function
$$
(K,x) \mapsto \int_{y\in Y_K} f_K(x,y) |dy|
$$
lies in $C(X)$ (with our usual product measure on $Y_K$).
\end{enumerate}
\end{prop}
\begin{proof}
We check each of the generators listed in Section \ref{sec:cC-func}.
The generator $q^{\beta}: X\to \RR$, of the kind (\ref{cC3}), comes up as the parameter integral of
$$
x\mapsto \int_{y\in K} \11_{ Y_K   }(x,y) |dy|,
$$
where $\11_{ Y_K   }$ is the characteristic function of the definable set $Y\subset X\times \VF$ with $Y_K=\{((x,y)\in X_K\times K\mid \ord y \geq - \beta(x ) \}$.

For generators of the kind (\ref{cC1}) it is enough to check that $\frac{1}{1-q^{-i}}$ is in $\cC(X)$ for $i>0$. To this end, let $Y$ be the definable subset of $X\times \VF$ such that $Y_K = \{(x,y)\in X_K\times K\mid \ac (y) = 1,\ \ord (y)\geq -1,\ \ord (y) \equiv -1 \bmod i  \}$
and consider
$$
x\mapsto \int_{y\in K} \11_{ Y_K   }(x,y) |dy|.
$$

If a generator $\alpha : X\to \RR$ of the kind (\ref{cC2}) of $\cC(X)$ takes only non-negative values, it is equal to
$$
x\mapsto \int_{y\in K} q_K^{ \ord y +1 } \11_{ Y_K   }(x,y)  |dy|,
$$
where $\11_{ Y_K   }$ is the characteristic function of the definable set $Y\subset X\times K$ with $Y_K=\{(x,y)\in X_K\times \VF\mid \ac(y)=1,\ 0 \leq \ord y <  \alpha(x)\}$.
General generators of the kind (\ref{cC2}) can be written as a difference of two of the above ones.

The generator $\# Y : X\to \RR$ with $Y\subset X\times \prod_{t=1}^{\ell} \Res_{n_{t}}$ as in (\ref{cC4}) comes up as
$$
x\mapsto \int_{z\in K^\ell } q_K^{ \sum_{t=1}^\ell 1+ \ord(n_t) } \11_{ Z_K   }(x,y)  |dz|,
$$
where $Z$ is the definable set such that
$$
Z_K=\{(x,z)\in X_K\times \cO_K^\ell\mid (x,\res_{n_1}(z_1),\ldots,\res_{n_\ell}(z_\ell)) \in Y \}.
$$
This proves the proposition and thus the naturality.
\end{proof}

\begin{prop}\label{prop:natcCexp}
The collection of the rings $\cCexp(X)$ for all definable sets $X$ is the smallest collection of rings of functions
$$
f: X \times \cD \to \CC
$$
which is stable under integration and with the properties that the characteristic function of any definable subset $A\subset X$ lies in $\cCexp(X)$ for any $A$ and $X$ and that $\psi(h)$, sending $K$ in $\Locp$, $x\in X_K$ and $\psi$ in $\cD_K$ to $\psi (h_K(x))$, lies in $\cCexp(X)$ for any definable function $h:X\to \VF$.
\end{prop}
\begin{proof}
We have to show that
\begin{equation}\label{genpsi}
(K,x)\mapsto \sum_{y \in Y_{K,x} } \psi\big( h_K(x ,y ) + \frac{e_K(x ,y )}{N}\big),
\end{equation}
lies in the above-mentioned smallest collection,
for a definable subset $Y$ of
$X\times \prod_{i=1}^\ell \Res_{n_i}$ for some $\ell\geq 0$ and some $n_i > 0$ and definable functions
$h:Y\to \VF$ and $e:Y\to \Res_N$ for some integer $N>0$. This is done as for generator (\ref{cC4}) in the proof of Proposition \ref{prop:natcC}.
Namely, consider
\begin{equation}\label{genpsi2}
(K,x)\mapsto \int_{z\in K^{\ell +1} } q_K^{1+\ord N + \sum_{t=1}^\ell 1+ \ord(n_t) } \psi ( g_K(x,z) + \frac{z_{\ell+1}}{N}  )  |dz|,
\end{equation}
where $Z$ is the definable set such that
$$
Z_K=\{(x,z)\in X_K\times \cO_K^{\ell +1 }\mid (x,\res_{\bullet}(z)) \in Y,
$$
$$
\res_{N}(z_{\ell+1}) =  e_K(x,\res_{\bullet}(z))\}
$$
and where $g$ is a definable function on $Z$ such that $g_K(x,z) = h_K (x, \res_{\bullet}(z))$, with
$\res_{\bullet}(z) = (\res_{n_1}(z_1),\ldots,\res_{n_{\ell}}(z_\ell)) $. Then (\ref{genpsi}) equals (\ref{genpsi2}) as required.
This proves the proposition and thus the naturality of $\cCexp$.
\end{proof}
Finally note that the functions of the form $\psi(h)$ in Proposition \ref{prop:natcCexp} can easily be created from functions of $\cC$-class by performing Fourier transformation as in Corollary \ref{stab:four:I} and change of variables as in Proposition \ref{cov}.

\subsection{}

The Fourier transform of an $L^1$ function of $\cCexp$-class is of $\cCexp$-class. (For the harder result about Fourier transform of $L^2$-functions of $\cCexp$-class, see \cite{CGH5}.) More generally, we have the following.

\begin{cor}[Stability under Fourier transformation]\label{stab:four:I}
Let $f$ be in $\cCexp(X\times \VF^m)$ for some $m\geq 0$ and some definable set $X$. Then there exists $\cF_{/X}(f)$ in $\cCexp(X\times \VF^m)$ such that the following holds for all $K$ in $\Locp$, all $\psi\in \cD_K$, and all $(x,z)\in X_K\times K^m$
$$
(\cF_{/X}(f))_{K,\psi}(x,z) = \int_{y\in K^m}f_{K,\psi}(x,y) \psi ( y\cdot z )|dy|,
$$
whenever $x\in X_K$ is such that $y\mapsto f_{K,\psi}(x,y)$ is integrable over $K^m$, and where $y\cdot z = \sum_{i=1}^m y_iz_i$.
\end{cor}

In contrast to the motivic integration from \cite{CLexp}, in our formalism,
several forms of change of variables formulas come for free by properties of the Haar measure on $K^m$ for $p$-adic fields $K$, and since partial derivatives of definable functions exist almost everywhere and are again definable when extended by zero where they do
not exist. One variant is the following, with our usual affine Haar measure on $K^m$.
\begin{prop}[Change of Variables]\label{cov}
Let $f$ be in $\cCexp(X\times \VF^m)$ for some $m\geq 0$ and some definable set $X$.
Suppose that
$$
H:Y'\subset X\times \VF^m \to Y\subset X\times \VF^m
$$
is a definable bijection over $X$ for some definable sets $Y'$ and $Y$. Then, for each $K\in \Locp$ and each $x\in X_K$ the Jacobian determinant $\Jac (H_{K,x})$ of
$$
H_{K,x}:Y'_{K,x}\to Y_{K,x}:y\mapsto H_K(x,y)
$$
is well-defined for almost all $y\in Y'_{K,x}$ (for the Haar measure on $K^m$) and one has for each $\psi\in\cD_K$
$$
\int_{Y'_{K,x}\subset K^m} q_K^{-\ord (\Jac (H_{K,x}))(y')} f_{K,\psi}\big(H_K(x,y')\big) |dy'| = \int_{y\in Y_{K,x}\subset K^m} f_{K,\psi}(x,y) |dy|.
$$
Moreover, the function $y'\mapsto \Jac (H_{K,x}))(y')$ is definable, when extended by $0$ when if it is not well defined at $y'$.
\end{prop}
\begin{proof}
The statements about $y'\mapsto \Jac (H_{K,x}))(y')$ are clear. The other statements follow from general measure theory on local fields.
\end{proof}

By virtue of Proposition \ref{cov} and the Cell Decomposition Theorem \ref{cd}, a dimension theory for definable sets and a theory of definable volume forms and their associated measures can be developed naturally and in analogy to the situation with large residue field characteristic.

\subsection{Loci}\label{sec:loci}

We generalize the results on loci of \cite{CGH} to the present setting. We provide full proofs for these generalizations, including for the key technical Proposition \ref{IML} below.
For a function $f:A\to \CC$, we write $Z(f)$ for the zero locus $\{a\in A\mid f(a)=0\}$ of $f$, and similarly for an $R$-valued function $f:A\to R$ for any ring $R$.
For arbitrary sets $A\subset X\times T$ and $x\in X$, write $A_x$ for the set of $t\in T$ with $(x,t)\in A$.
For $g:A\subset X\times T\to B$ a function
and for $x\in X$, write $g(x,\cdot)$ for the function $A_x\to B$ sending $t$  to $g(x,t)$.

\begin{defn}\label{loci}
Let $T$ and $X$ be arbitrary sets, and let $f:X\times T\to \CC$ be a function.
Define the \emph{locus of boundedness of $f$ in $X$} as the set
$$
\Bdd(f,X) := \{x\in X\mid f(x,\cdot) \mbox{ is bounded on }T\}.
$$
Define the \emph{locus of identical vanishing of $f$ in $X$} as the set
$$
\Iva(f,X) := \{x\in X\mid f(x,\cdot) \mbox{ is identically zero on }T\}.
$$
If moreover $T$ is equipped with a complete measure, we define the \emph{locus of integrability of $f$ in $X$} as the set
$$
\Int(f,X) := \{x\in X\mid f(x,\cdot) \mbox{ is measurable and integrable over }T\}.
$$
\end{defn}

The following result extends \cite[Theorem~4.4.4]{CGH} and its corollary \cite[Corollary~4.4.5]{CGH} to our more general setting.

\begin{thm}[Correspondences of loci and extrapolation]\label{p2pexpmot}
Let $f$ be in $\cCexp(X\times Y)$ for some definable sets $X$ and $Y$.
Then there exist $h_i\in \cCexp(X)$ for $i=1,\ldots,5$, such that, for all $K$ in $\Locp$ and
for each $\psi\in\cD_K$, the zero locus of $h_{i,K,\psi}$ in $X_K$ equals respectively
$$
\Int(X_K, f_{K,\psi}),
$$
$$
\Bdd(X_K, f_{K,\psi}),
$$
  $$
\Iva(X_K, f_{K,\psi}),
$$
$$
\{x\in X_K\mid f_{K,\psi}(x,\cdot) \mbox{ is locally integrable on }Y_K\},
$$
$$
\{x\in X_K\mid f_{K,\psi}(x,\cdot) \mbox{ is locally bounded on }Y_K\}.
$$
Moreover, there exist (``extrapolating'') functions $g_i$ in $\cCexp(X\times Y)$ such that for each $K$, each $x$, and each $\psi$, the function $y\mapsto g_{i,K,\psi}(x,y)$ on $Y_K$ is integrable, resp.~bounded, identically vanishing, locally integrable, locally bounded, for $i=1,\ldots,5$ respectively, and such that $g_{i,K,\psi}(x,y) = f_{K,\psi}(x,y)$ whenever the function $y \mapsto f_{K,\psi}(x, y)$ on $Y_K$ satisfies the condition corresponding to $i$.
\end{thm}
We give a full proof of this theorem, exhibiting similar modifications to the large residue field characteristic case treated in \cite{CGH} as done for the proof of Theorem~\ref{p1qmotexp}.

\begin{proof}[Proof of Theorem \ref{p2pexpmot} for existence of $h_3$ and $g_3$]
It is enough to treat the cases that $Y$ is $\RF_n$ for some $n>0$, $\VG$, or $\VF$. For $Y=\RF_n$, the result follows from Theorem \ref{p1qmotexp}, as follows. By changing the sign of all the arguments of $\psi$ appearing in the build-up up $f$ and taking the product with $f$, we obtain a function $f_1$ in $\cCexp(X\times Y)$ such that
$$
f_{1,K,\psi}(x,y) = |f_{K,\psi}(x,y)|_\CC^2
$$
for all $K$ and all $\psi$. Now let $g$ be obtained from $f_1$ and Theorem \ref{p1qmotexp}, by summing out the $\RF_n$-variable. Then the zero locus of $g$ is as desired for $h_3$.

Secondly, suppose that $Y$ is $\VG$. As in the case $Y=\VG$ in the proof of Theorem \ref{p1qmotexp}, we may suppose that $f_K$ is supported on a set of the form as in (\ref{eq:W}), and that $f$ is a finite sum of terms like $T_i$ as in (\ref{eq:Ti}) with mutually different pairs $(a_i,b_i)$ for different $i$. If $\sq$ is no condition, then, by observing different asymptotic behavior of the terms $T_i$ for different $i$, the sum of the squares of the complex moduli of the $c_i$ can serve as $h_3$. If $\sq$ is $<$ then the sum (over $y$) of squares of complex moduli of $f_{K,\psi}(x,y)$ is of $\cCexp$-class and is as desired for $h_3$.

Finally, the case that $Y=\VF$ (which is the hardest case) is reduced to the two cases treated above using Proposition \ref{IML}: The proposition allows us to write $f$ as in (\ref{eq:expSum}) in such a way that
we can treat each summand individually (taking as final $h_3$ the sum of the squares of the complex norms of the $h_3$ corresponding to the individual summands),
and for individual summands, we may, using Theorem~\ref{cd} and after introducing new $\VG$- and $\RF_n$-variables, assume that $f(x,y)$ only depends on $x$. By the already treated cases of $Y$ being $\VG$, or $\RF_n$, or a product of these, we are done for $h_3$. The construction of $g_3$ is trivial in this case, since one can simply take the zero function in $\cCexp(X\times Y)$.
\end{proof}

\begin{proof}[Proof of Theorem \ref{p2pexpmot} for existence of $h_1$ and $g_1$]
We will work with the more general situation that $f$ is a function in $\cCexp(V)$ for some definable subset $V$ of $X\times Y$, where we ask whether $f_{K,\psi}(x, \cdot)$ is integrable on the corresponding fiber $V_{K,x}$. We use Proposition \ref{IML} (for general $m\geq 0$), to simplify the shape of $Y$. By that proposition, one reduces to the case that $Y$ is a definable subset of a Cartesian product of $\prod_{i=1}^D \Res_{r_i}\times \VG^t$ for some $D\geq 0$, $r_i>0$, and $t\geq 0$.
Moreover, by Proposition \ref{prop:recti}, by working piecewise and up to performing a change of variables (for the counting measure on the value group and on the $RF_n$, hence, without correction by the determinant of a Jacobian), we may suppose, for each $K$ and each $x\in X_{K}$, that $V_{K,x}$ has the form
$$
\Lambda_{K,x}\times \NN^\ell
$$
for a finite definable subset
$$
\Lambda_{K,x}\subset (\prod_{i=1}^D \Res_{r_i,K}) \times \NN^{t-\ell}
$$
which may depend on $x$, but where the  integer $\ell\geq 0$ does not depend on $x\in X_K$ and neither on $K$ (and neither do $D$ and the $r_i$). Since integrability of $f_{K,\psi}(x,\cdot)$ over $V_{K,x}$ is equivalent to integrability of $f_{K,\psi}(x,\lambda,\cdot)$ over $\NN^\ell$ for each $\lambda\in \Lambda_{K,x}$ (by the finiteness of $\Lambda_{K,x}$), and by the already proved existence of $h_3$ for $\Iva$ of the theorem, we may suppose that $\Lambda_{K,x}$ is absent, that is,
$$
V_{K,x} = \NN^\ell
$$
for each $K$ and each $x\in X_K$. Indeed, a function $g_3$ in $\cCexp(V)$ which extrapolates integrability of the function $f_{K,\psi}(x,\lambda,\cdot)$ will also extrapolate integrability of $f_{K,\psi}(x,\cdot)$ as desired. From now on we suppose thus that $V_{K,x}=\NN^\ell$ for each $K$ and $x$.
By our application of Proposition \ref{prop:recti} and subsequent simplification of $V_{K,x}$, we may as well suppose that the definable functions which take values in $\VG$ and which appear in the build-up of $f$ (namely in the forms of generators (\ref{cC2}) and (\ref{cC3}) of  Section \ref{sec:cC-func}) are linear over $X$, and that all other build-up data of $f$ (namely, generators  (\ref{cC4}) of  Section \ref{sec:cC-func} and $h$ and $e$ as in (\ref{eq:gen5})) factor through the projection to $V\to X$. Now, for any $K$ in $\Locp$, any $x\in X_K$, and $y\in V_{K,x} = \NN^\ell$ and any $\psi$ in $\cD_K$, the function $(x,y)\mapsto f_{K,\psi}(x,y)$  is a finite sum of terms $T_i$ of the form
\begin{equation}\label{eq:Tit1}
c_{i,K,\psi}(x)y^{a_i}q_K^{b_iy},
\end{equation}
with multi-index notation, for a tuple of nonnegative integers $a_i=(a_{ij})_{j=1}^\ell$, a tuple of rational numbers $b_i=(b_{ij})_{j=1}^\ell$ and  $c_i$ in $\cCexp(X)$. By regrouping we may suppose that $(a_i,b_i)$ is different from $(a_{i'},b_{i'})$ when $i\not=i'$. Let
\begin{equation}\label{eq:I1}
I
\end{equation}
be the set of those $i$ such that $b_{ij}\geq 0$ for some $j$. For $x\in X_K$ and $\psi\in \cD_K$, the function
$$
y  \mapsto f_{K,\psi}(x,y)
$$
is integrable over $V_x=\NN^\ell$, if and only if
$$
c_{i,K,\psi}(x)=0
$$
for each $i$ in $I$. Hence, as the extrapolating function we can take
$$
g_{1,K,\psi}(x,\xi,y) = \sum_{i\not\in I} c_{i,K,\psi}(x,\xi)y^{a_i}q_K^{b_iy}.
$$
Now let $h_1$ be the function in $\cCexp(X)$ so that $h_{1,K,\psi}$ has as zero locus precisely $\Iva(g_{1,K,\psi} - f_{K,\psi},X_K)$, which exists by the already proved case for $\Iva$ of Theorem \ref{p2pexpmot}, and which is as desired by the properties of $g_1$.
\end{proof}

\begin{proof}[Proof of Theorem \ref{p2pexpmot} for existence of $h_2$ and $g_2$]
The proof is the same as for $h_1$ and $g_1$, now taking instead for $I$ in (\ref{eq:I1}) the set of those $i$ such that, for some $j$ one has either $b_{ij}=0$ and $a_{ij}>0$, or, one has $b_{ij}> 0$.
\end{proof}

\begin{proof}[Proof of Theorem \ref{p2pexpmot} for existence of $h_i$ and $g_i$ for $i=4$ and $i=5$]
Given $f$ as in the theorem, define a function $\tilde f$ in $\cCexp(X\times Y\times Y)$ such that
$$
\tilde f_{K,\psi}(x,y,y') = f_{K,\psi}(x,y) \11_{C_{K,y'}}(y)
$$
where $\11_{C_{K,y'}}$ is the characteristic function of the fiber $C_{K,y'}\subset Y_K$ (for some fixed definable set $C \subset X \times Y$)
which is a compact neighborhood of $y'\in Y_K$ (for the product topology where the discrete topology is put on $\ZZ$ and on the $\RF_{n,K}$, the ultrametric topology on $K$, and the induced subset topology on $Y_K$). Now take $g_1,g_2,h_1,h_2$ as given by the already proved part of the theorem, but for the function $\tilde f$ instead of for $f$ and with with $(x,y') \in X \times Y$ in the role of $x \in X$. Finally take  $g_4$ and $g_5$ in $\cCexp(X\times Y)$ such that
$$
g_{4,K,\psi}(x,y) = g_{1,K,\psi}(x,y,y)
$$
and
$$
g_{5,K,\psi}(x,y) = g_{2,K,\psi}(x,y,y).
$$
Then $g_4$ and $g_5$ are the desired extrapolating functions.  Now, for $i=4$ and $i=5$, let  $h_i$ be the function in $\cCexp(X)$ so that $h_{i,K,\psi}$ has as zero locus precisely $\Iva(g_{i,K,\psi} - f_{K,\psi},X_K)$, which exists by the already proved case for $\Iva$ of Theorem \ref{p2pexpmot}, and which is as desired by the properties of $g_i$.
\end{proof}

One key property that remains preserved when omitting the bound on ramification, is that several kinds of bad behavior or bad loci in the valued field are typically contained in a small definable set (for example of lower dimension). Moreover, there is typically a single proper, Zariski closed subset which is given independently of the residual characteristic and which captures bad loci uniformly in the local field.

The following theorem is an example result that `bad loci' are uniform and have a geometrical nature. Note that it in particular applies to the functions $h_i$ from Theorem \ref{p2pexpmot}.
\begin{thm}\label{badlocus}
Let $f$ be a function in $\cCexp(\VF^n)$. Then, there is a proper Zariski closed subset $C$ of $\AA^n_\QQ$ such that the locus of non-local-constancy of $f_{K,\psi}$ is contained in $C(K)$ for each $K$ in $\Locp$ and each $\psi$ in $\cD_K$. More precisely, for each $K$ in $\Locp$ and each $\psi$ in $\cD_K$, the set
$$
\{x\in K^n\mid f_{K,\psi} \mbox{ is not locally constant around $x$} \}
$$
is contained in $C(K)$, the set of $K$-rational points on $C$.
\end{thm}
\begin{proof}
By the quantifier elimination \ref{QEZ}, all definable functions $g_i$ appearing in the build up of $f$ according to the definition of $\cCexp$, as well as their multiplicative inverses $1/g_i$ extended by zero on $g_i=0$, have the property that there is a proper Zariski closed subset $C_i$ of $\AA_\QQ^n$ such that the locus of non-continuity of $g_{iK}$ and of $1/g_{iK}$ is contained in $C_i(K)$ for each $K$ in $\Locp$. Now the theorem follows from the definition of functions of $\cCexp$-class.
\end{proof}

Note that family variants of Theorem \ref{badlocus} can be obtained with essentially the same proof.


\subsection{Local zeta functions}\label{loc:zeta}
Local zeta functions and their poles have strong uniformity properties when the $p$-adic field varies.
Corollary \ref{thm:rational} is an example of how the present framework can describe the uniform behaviour of local zeta functions, the novelty being the combination of allowing small primes and not bounding the degree of ramification for small primes. It will be shown to follow from Theorem \ref{thm:cCs}.

Let us first introduce two new classes of functions, for use in this section only.By a function on $\RR_{\gg 0}$ we mean a function well-defined for sufficiently large real input values $s$, also written as $s\gg 0$.
For a definable set $X$, by $\cCs(X)$ we denote the ring of finite sums and products of functions of the form
\begin{enumerate}
\item[(i)]\label{cCs1} $q^{\beta s} : X\times \RR_{\gg 0}\to \RR$, for any $\VG$-valued definable function $\beta$, sending input $K\in \Locp$, $x\in X_K$ and real $s\gg 0$ to  $q_K^{\beta_K(x)s}$.
\item[(ii)]\label{cCs2} $\frac{1}{1-q^{a+bs}}: X\times \RR_{\gg 0} \to \RR$ for any integers $a$ and nonzero $b$, sending input $K\in \Locp$, $x\in X_K$ and $s\gg 0$ to  $\frac{1}{1-q_K^{a+bs}}$.
\item[(iii)]\label{cCs3} $f : X\times \RR_{\gg 0}\to \RR$, for any $f\in \cC(X)$, sending sending input $K\in \Locp$, $x\in X_K$ and $s\gg 0$ to  $f_{K}(x)$.
\end{enumerate}
For a definable set $X$, let $\cCexps(X)$ be the ring of finite sums and products of functions in $\cCs(X)$ and in $\cCexp(X)$.

\begin{thm}[Stability under Integration with complex powers]\label{p1qmotexps}\label{thm:cCs}
Let $f$ be in $\cCexps(W)$ for some definable sets $X$, $Y$ and $W\subset X\times Y$. Then there exists $g$ in $\cCexps(X)$ such that the following holds for each $K$ in $\Locp$, each $x\in X_K$, and each $\psi\in \cD_K$:
If, for all $s \gg 0$, the function $y\mapsto f_{K,\psi}(x,y,s)$ is integrable over $W_{K,x}$, then
for all $s\gg 0$,
$$
g_{K,\psi}(x,s) = \int_{y\in W_{K,x}}f_{K,\psi}(x,y,s),
$$
Furthermore, if $f$ lies in $\cCs(W)$, then $g$ can be taken in $\cCs(X)$.
\end{thm}
\begin{proof}
The proof is very similar to the one for Theorem~\ref{p1qmotexp}.
By Fubini's Theorem, it is enough to treat the case that $Y$ is either $\VF$, $\VG$ or $\RF_{n}$ for some integer $n>0$. The only case which is different and needs special attention is when $Y$ is $\VG$.
Let us again first suppose that the definable functions which take values in $\VG$ and which appear in the build-up of $f$ (namely in the forms of generators (\ref{cC2}) and (\ref{cC3}) of Section \ref{sec:cC-func}), and of generator (i) above, are linear over $X$, that $W$ is the definable set
\begin{equation}\label{eq:W:s}
\{(x,y)\in X\times \VG\mid \alpha(x) \leq y \sq \beta(x)\}
\end{equation}
where $\sq$ is $<$ or no condition and where $\alpha,\beta:X\to\VG$ are definable functions, and that all other build-up data of $f$ (namely, generators  (\ref{cC4}) of  Section \ref{sec:cC-func} and $h$ and $e$ as in (\ref{eq:gen5})) factor through the projection $W\to X$. Then the conclusion follows from Lemma \ref{trl}.

Indeed, for any $K$ in $\Locp$, any $x\in X_K$, any real $s>0$, and any $\psi$ in $\cD_K$, $f_{K,\psi}(x,y)$ is a finite sum of terms $T_i$ of the form
\begin{equation}\label{eq:Tis}
c_{i,K,\psi}(x,s) y^{a_i}q_K^{(b_i + b'_i s)y}
\end{equation}
for integers $a_i\geq 0$, rational numbers $b_i$ and  $b'_i$, and with $c_i$ being a sum of products of generators of the form (ii) and (iii) above. The integrability of $f_{K,\psi}(x,y)$ over $y$ in $W_{K,x}$ is automatic when $\sq$ is $<$ and we get $g$ from Lemma \ref{trl}. When $\sq$ is no condition, we regroup the terms if necessary, so that the pairs $(a_i,b_i,b'_i)$ are mutually different for different $i$. By observing different asymptotic behavior of these terms for growing $y$ for any $s\gg 0$, we may consider the sub-sum $\sum_{i\in J} T_i$ with $i\in J$ if and only if $b_i + b'_i s < 0$ for all $s\gg 0$. This time we apply Lemma \ref{trl} to this sub-sum to find $g$.
The Presburger results from Section \ref{sec:presburgering} (namely Corollary \ref{linear} and Proposition \ref{prop:recti}) together with the case that $Y= \RF_{n}$ finish the general case that $Y = \VG$.
\end{proof}

Now we come to our uniform rationality application, following from stability under integration of complex powers given by the previous theorem.

\begin{cor}[Uniform rationality]\label{thm:rational}
Let $f$ be a $\VG$-valued definable function on a definable set  $X\subset \VF^n$. Suppose that, for real $s \gg 0$,
$$
Z_K(s) := \int_{x\in X_K\subset \cO_K^n} q_K^{s f_K(x)}  |dx|
$$
is finite for each $K\in \Locp$. Then there are an integer $b$, a nonzero $c\in\QQ$, and a finite collection of pairs of integers $(a_i,b_i)$ with $b_i\not=0$, for $i=1,\ldots,N$ for some $N$, such that for any $K$ in $\Locp$,
\begin{equation}\label{eq:denom}
Z_K(s) q_K^{(b+\ord c) s} \prod_{i=1}^N(1-q_K^{a_i+b_is})
\end{equation}
is a polynomial in $q_K^{-s}$.
\end{cor}

\begin{proof}
By Theorem \ref{p1qmotexps}, we just have to look at finite sums and products of generators of $\cCs(*)$, where $*$ is the point (e.g. the definable set $\{0\}$). But they are clearly rational in $q_K^{-s}$ after multiplying by something of the form $q_K^{(b+\ord c) s} \prod_{i=1}^N(1-q_K^{a_i+b_is})$
as desired. (Note that indeed, the $\beta$ from the generator (i) above
can be bounded uniformly for all $K$ by some $b+\ord c$.)
\end{proof}

A more refined description for the numerator of $Z_K(s)$ of elements of $\cCs(*)$, where $*$ is the point, is also possible (in terms of cardinalities
of definable sets), by looking at the possible generaotors $\cCs(*)$.

A proof for Corollary \ref{thm:rational} using resolution of singularities and Theorem \ref{QERid} may also be thinkable, but we prefer to use the more general Theorem \ref{thm:cCs} which uses the full strength of the results of Section \ref{sec:qe}.

\subsection{A result behind bounds, integrability, and loci}

The following generalizes the key technical Proposition 4.5.8 of \cite{CGH} to our setting. It lies behind the deeper aspects of the results of Section \ref{subsec:int} and \ref{sec:loci}. Roughly, Proposition \ref{IML} with $s=1$ says that if $|f_{K,\psi}|_\CC$ is small for some $\cCexp$-function $f$, then $f$ is the
sum of small terms of a very specific form. More precisely, if $f$
cannot be written as a sum of small terms as in Proposition \ref{IML}(1), then $|f_{K,\psi}|_\CC$ has to
be large on a relatively large set, namely, on the set $W_{K,\psi,x,r}$.
In particular, $f$ is integrable resp.\ bounded resp.\ identically zero if and only if all summands are.

\begin{prop}
\label{IML} 
Let $m\geq 0$ and $s\geq 0$ be integers, let $X$ and $U \subset X \times \VF^m$ be definable,
and let $f_1,\ldots,f_s$ be in $\cCexp(U)$. Write $x$ for variables running over $X$ and $y$ for variables running over $\VF^m$.
Then there exist integers $N_\ell\geq 1$, $d\geq 0$, $D\geq 0$, $r_i\geq 0$, $t\geq 0$, a definable surjection $\varphi:U\to V\subset X\times \prod_{i=1}^D \Res_{r_i}\times \VG^t$
over $X$, definable functions $h_{\ell, i}: V \times \VF^m \to \VF$, and functions $G_{\ell, i}$ in $\cCexp(V)$ for $\ell=1,\ldots,s$ and $i=1,\ldots,N_\ell$,
such that the following conditions hold for each $K\in \Locp$ and each $\psi\in \cD_K$.
 \begin{itemize}
 \item[1)] One has
 $$
 f_{\ell,K,\psi}(x,y) = \sum_{i=1}^{N}G_{\ell, i,K,\psi}(\varphi_K(x,y))\psi(h_{\ell,i,K}(\varphi_K(x,y),y));
 $$
 \item[2)] if one sets, for $(x,r)\in V_K$,
$$
U_{K,x,r} := \{y\in U_{K,x}\mid \varphi_K(x,y)=(x,r)\}
$$
 and
 $$
 W_{K,\psi,x,r}:=\{y \in U_{K,x,r}\mid
 \sup_{\ell, i} | G_{\ell, i,K,\psi}(x,r)|_{\CC} \leq \sup_{\ell} |f_{\ell, K,\psi }(x,y)|_{\CC} \},
 $$
 then
\[
\Vol(U_{K,x,r}) \leq q_K^d \cdot \Vol(W_{K,\psi,x,r}) < +\infty,
\]
where the volume $\Vol$ is taken with respect to the Haar measure on $K^m$.
\end{itemize}
\end{prop}
We first generalize Lemma 3.3.6 of \cite{CGH} to our setting.
\begin{lem}\label{avoidSkol}
Let $A\subset X\times \VF$ and $h:A\to \VF$ be definable for some definable set $X$. Suppose that for each $K$ in $\Locp$, each $x\in X_K$, and for each ball $B$ contained in $A_{K,x}$, the function $h_K(x,\cdot)$ is constant modulo $(\varpi_K)$ on $B$. Then there exist positive integers $m$, $n$, a definable function
$$
\lambda:A\to A'\subset \RF_n^n\times X \mbox{ over } X
$$
and a definable function $h':A'\to \VF$ such that, for each $K$ in $\Locp$, for each $(x,y)\in A_K$ and with $(x,r)=\lambda_K(x,y)$, one has
$$
|h_K(x,y)-h'_K(x,r)| \leq 1/|m|.
$$
\end{lem}
\begin{proof}
By Theorem \ref{cd} and with its notation, there is a definable function
$$
\sigma:A\to A_{\rm par}\subset \Res_n^n\times A
$$
over $A$ onto a presented cell $A_{\rm par}$ over $\Res_n^n\times X$ such that each $h_{\rm par}$ has the $1$-Jacobian property over $\Res_n^n\times X$. Let $G$ be the graph of $h_{\rm par}$, and let $W$ be the image of $G$ under the coordinate projection $(r,x,y,h(x,y)) \to (r,x,h(x,y))$, where $x$ runs over $X$ and $y$ over $\VF$. We may suppose that $W$ is also a presented cell over $\Res_n^n\times X$, again by Theorem \ref{cd}, say, with center $c$ and depth $m$. Now take $\lambda$ to be $\sigma$ composed with the coordinate projection to $\RF_n^n\times X$, and take $h'=c$. Then $m,n,\lambda,h'$ are as desired.
Indeed, the condition about $h$ being constant modulo $(\varpi_K)$ on balls implies (together with $h_{\rm par}$ having the $1$-Jacobian property) that each ball contained in some $W_{r,x}$ has at most the volume of the maximal ideal.
\end{proof}

\begin{proof}[Proof of Proposition~\ref{IML} for $m=1$]
The statement that we have to prove allows us to work
piecewise; if we have a finite partition of $U$ into definable parts $A$, then it suffices to prove the proposition for $f_\ell$ restricted to each part $A$. We actually prove something slightly stronger than
Proposition~\ref{IML} for the case $m=1$. That is, for a given definable function $\varphi_0:U\to X\times \prod_{i=1}^{D_0} \Res_{r_{i0}}\times \VG^{t_0}$ over $X$, we prove that in addition to the conclusions 1) and 2) of the
proposition, we can require that also the following conditions 3) and 4) hold.
\begin{itemize}
\item[3)]
For each $K\in \Locp$, each $x\in X_K$ and each $r$ with $(x,r)\in V_K$, the set $U_{K,x,r}$ is either a singleton or a ball.
\item[4)]\label{cond4} The function $\varphi_0$ factors through $\varphi$, that is,  $\varphi_0= \theta\circ \varphi$ for some definable function $\theta$.
\end{itemize}

So, let a definable function $\varphi_0:U\to V_0\subset X\times \prod_{i=1}^{D_0} \Res_{r_{i0}}\times \VG^{t_0}$ over $X$ be given.
By definition of $\cCexp$, Theorem \ref{cd}, and by replacing $\varphi_0$ by a definable function through which the original $\varphi_0$ factors, we may suppose
that there are definable functions $h_{\ell, i}:V_0\times \VF \to \VF$ and functions $G_{\ell, i}$ in $\cCexp(V_0)$ such that for each $\ell$, for each $K\in \Locp$, each $\psi \in \cD_K$, each $(x,y)\in U_K$ with $x\in X_K$, and each $r$ with $(x,r)\in V_K$ one has
\begin{equation}\label{f-as-sum}
f_{\ell,K,\psi}(x,y) = \sum_{i=1}^{N_\ell}G_{\ell, i,K,\psi}(\varphi_{0,K}(x,y))\psi(h_{\ell, i,K}(\varphi_{0,K}(x,y),y)),
\end{equation}
and that the set $U^{0}_{K,x,r} := \{y\in U_{K,x} \mid \varphi_0(x,y)=(x,r)\}$ is either a singleton or a ball. Thus, we may suppose that the conditions 1), 3) and 4) already hold for $\varphi_0$.
We now construct $\varphi$ (and modify $G_{\ell, i}$ and $h_{\ell, i}$ accordingly) such that moreover 2) holds.  

We will proceed by induction on $N := \sum_{\ell=1}^s (N_\ell - 1)$.
Namely, fix $N$ and assume that for any finite family of functions $\{f_\ell\}$ on a definable set $U$ (not necessarily the same family and the same set as the given one), such that
the functions $f_\ell$ have a presentation of the form (\ref{f-as-sum}) and satisfying the properties 1), 3), and 4), and with $\sum (N_{\ell}-1)<N$, there exists a function $\varphi$ such that the property 2) holds as well. Then we want to prove the same for any such family and presentation with $\sum(N_\ell-1)=N$. The idea of the proof of the induction step is to increase the number of functions in the family without increasing the total number of terms in their presentations (\ref{f-as-sum}), and thus decrease
$\sum(N_\ell -1)$.
Note that the constant $d$ appearing in 2) will increase by
at most $1$ in each induction step, so that we actually obtain $d \le N$.

If $N = 0$, then all $N_\ell = 1$, and one is done, taking $\varphi=\varphi_0$ and $d = 0$. Indeed, if $N_\ell = 1$, then $|G_{\ell, 1,K,\psi}(x,r)|_{\CC}$ equals $|f_{\ell,K,\psi}(x,y)|_{\CC}$, and thus, if $N=0$, then  $U_{K,x,r}=W_{K,\psi,x,r}$.

For general $N>0$ we start by pulling out the factor $\psi(h_{\ell, 1})$ out of (\ref{f-as-sum}), i.e., we may assume that $h_{\ell, 1,K}=0$ for all $\ell$ and all $K$.
By Theorem \ref{cd} we may moreover suppose that for each $K$, $(x,r)\in V_{0,K}$, $\ell$, and each $i$, either $h_{\ell, i,K}(\varphi_0(x,y),\cdot)$ is constant on $U^0_{x,r}$, or, $h_{\ell, i,K}(x,\cdot)$ restricted to  $U^0_{x,r}$ has the $1$-Jacobian property.
Hence, for each $K$ and $(x,r)\in V_{0,K}$ there exist constants $b_{K,x,r,\ell,i}\in  K$ such that, for all $y_1,y_2\in U^0_{K,x,r}$ and all $\ell,i$, and with $z=\varphi_{0}(x,y_1)$,
\begin{eqnarray}
 \ord(h_{\ell,i,K}(z,y_1) - h_{\ell,i,K}(z,y_2))
 & = &
 \ord(b_{K,x,r,\ell,i}\cdot (y_1-y_2)), \label{jac1} \\
 \ac (h_{\ell,i,K}(z,y_1) - h_{\ell,i,K}(z,y_2))
 &=&
 \ac(b_{K,x,r,\ell,i}\cdot (y_1-y_2)) \label{jac2},
\end{eqnarray}
where $b_{K,x,r,\ell,1} = 0$ by a previous assumption.
If for all $K,\ell, i, x,r$, the function $h_{\ell,i,K}(x,\cdot)$ is constant modulo $(\varpi_K)$ on $U^0_{K,x,r}$, then, up to refining the function $\varphi_0$, Lemma \ref{avoidSkol} applied to each of the $h_{\ell,i,K}$ brings us back to the case  $N=0$.
Indeed, for each $h = h_{\ell,i}$ we refine $\varphi_0$ using the maps $\lambda$ and $\ac_m(h - h')$ (in the notation of Lemma \ref{avoidSkol});
after that, $\varphi_0$ determines $\psi \circ h$ and for each $\ell$, we can incorporate the entire sum (\ref{f-as-sum}) into a single $G_{\ell,1}$.

We may thus in particular assume that for each $K$ and each $(x,r)$ in $V_{0,K}$, there exist $\ell, i$ with $b_{K,x,r,\ell,i} \ne 0$. Choose $\gamma_{K,x,r} \in K$ with
\[
|\gamma_{K,x,r}|\cdot \max_{\ell,i} |b_{K,x,r,\ell,i}| = 1.
\]
For each $K$, $x$, $r$ and $\ell$, partition $\{1, \dots, N_\ell\}$ into non-empty subsets $S_{K,\ell, j}(x,r)$, $j \ge 1$, with the property that $i_1,i_2$ lie in the same part $S_{K,\ell, j}(x,r)$ for some $j$ if and only if
\begin{equation}\label{res-equal}
\res (\gamma_{K,x,r}b_{K,x,r,\ell,i_1})  = \res (\gamma_{K,x,r}b_{K,x,r,\ell,i_2}),
\end{equation}
where $\res:\cO_K\to k_K$ is the natural projection.
By cutting $U$ into finitely many pieces again,
we may assume
that the sets $S_{\ell, j}:=S_{K,\ell, j}(x,r)$ do not depend on $K$ nor on $(x,r)$.
Since $b_{K,x,r,\ell,1} = 0$, at least for one $\ell$ there are at least two different sets
$S_{\ell,j}, S_{\ell,j'}$.
Define for each $K, \psi, \ell, j$ and for $(x,y)\in U_K$
$$
f_{\ell, j,K,\psi}(x,y) := \sum_{i\in S_{\ell, j}} G_{\ell, i,K,\psi}(\varphi_0(x,y)) \psi(h_{\ell, i,K}(\varphi_0(x,y),y))
$$
and consider these functions $(f_{\ell, j})_{\ell,j}$ as a single family.
The total number of summands of the family $(f_{\ell, j})_{\ell,j}$ is the same
as for the functions $f_\ell$, but there are more functions $f_{\ell, j}$ than $f_j$, so
we can apply induction on $N$ to this family $(f_{\ell, j})_{\ell,j}$, with the extra conditions 3) and 4) for $\varphi_0$ as part of the desired properties.
Thus we find an integer $d\geq 0$, a definable surjection $\varphi:U\to V$ over $X$, definable functions $h_{\ell, j,i}:V \times \VF\to K$, and functions $G_{\ell, j,i}$ with properties 1), 2), 3) and 4) for $\varphi_0$ and for this family.

Let us write $U_{K,x,r}$ for the sets defined by $\varphi$ as in condition 2). Since
$\varphi_0=\theta\circ \varphi$ for some definable $\theta$, one has $U_{K,x,r}\subset U^0_{K,x,r'}$ for each $(x,r)$ and $(x,r')= \theta_{K}(x,r)$. 
By cutting $U$ into pieces as before, we may assume that, for each
$K$, $x$ and $r$,
not all $h_{\ell, i,K}(x,\cdot)$ are constant modulo $(\varpi_K)$ on $U_{K,x,r}$, since, as before, this would bring us back to the case $N=0$ for our original family $(f_\ell)_\ell$.

We will now show that the subset $M_{K,\psi,x,r}$ of $U_{K,x,r}$ consisting of those $y$ satisfying both inequalities
\begin{equation}\label{goal-iml}
\sup_{\ell,j,i} |G_{\ell, j,i,K,\psi}(x,r)|_{\CC}\leq \sup_{\ell,j} |f_{\ell, j,K,\psi}(x,y)|_{\CC} \leq \sup_\ell  |f_{\ell,K,\psi}(x,y)|_{\CC}
\end{equation}
has big volume in the sense that
\begin{equation}\label{vol-goal}
\Vol( U_{K,x,r}) \le q_K^{d+1}  \Vol(M_{K,\psi,x,r}).
\end{equation}

Once this is proved, we are done for our original family $(f_\ell)_\ell$ by replacing $d$ with $d+1$ while keeping the data of the $\varphi$, $G_{\ell, j,i}$, and $h_{\ell, j,i}$.

Thus, to finish the proof, we fix $K$, $\psi$, $x$ and $r$ and it remains to
show that $M_{K,\psi,x,r} $ as given by (\ref{goal-iml}) has the property (\ref{vol-goal}).
Consider the partition of the ball $U_{K,x,r}$ into the balls $B_\xi$ of the form $\xi+\gamma_{K,x,r}\cO_K$.
(The ball $U_{K,x,r}$ is indeed a union of such balls $B_\xi$ by our choice of $\gamma_{K,x,r}$ since there exists a $h_{\ell, i,K}(x,\cdot)$ that is non-constant modulo $(\varpi_K)$
on $U_{K,x,r}$.) Firstly we will show that $|f_{\ell, j,K,\psi}(x,\cdot)|_{\CC}$ is constant on each such
$B_\xi$. Secondly we will show that for each such $B_\xi$ there is a sub-ball $B'_\xi \subset B_\xi$ with $\Vol(B_\xi) = q_K\cdot \Vol(B'_\xi)$ and such that
the second inequality of (\ref{goal-iml}) holds for all $y\in B'_\xi$. These two facts together with the previous application of the induction hypothesis imply (\ref{vol-goal}) and thus finish the proof for $m=1$. Fix $B_\xi\subset U_{K,x,r}$
and write
$y = \xi + \gamma_{K,x,r}y' \in B_\xi$ for $y' \in \cO_K$.
By (\ref{jac1}), (\ref{jac2}), and (\ref{res-equal}), for each $\ell$ and $j$ there is a constant
$c_{\ell, j} \in \CC$ such that
\[
f_{\ell, j,K,\psi}(x,y) = c_{\ell, j} \psi(b'_{K,\ell, j}y'),
\]
where we can take $b'_{K,\ell, j} = \gamma_{K,x,r}b_{K,x,r',\ell,i}$ for any $i \in S_{\ell, j}$ where $r'$ is such that $U_{x,r}\subset U^0_{x,r'}$.
This shows that $|f_{\ell, j,K,\psi}(x,\cdot)|_{\CC}$ is constant on
$B_\xi$. We now only have to construct $B'_\xi$.
By renumbering, we can suppose that on $B_\xi$,
$|f_{1,1,K,\psi}|_{\CC}$ is maximal among the $|f_{\ell, j,K,\psi}|_{\CC}$,
so that the middle expression of (\ref{goal-iml}) is equal to $|f_{1,1,K,\psi}|_{\CC}$.
In particular, it suffices to choose $B'_\xi$ such that $|f_{1,1,K,\psi}(x,y)|_{\CC} \le |f_{1,K,\psi}(x,y)|_{\CC}$
for all $y \in B'_\xi$.
 Now let $\psi_q$ be the additive character of $\FF_{q_K}$ satisfying
$\psi(y') = \psi_q(\res(y'))$ for $y' \in \cO_K$.
By (\ref{res-equal}), we have $\res(b'_{K,1,j}) \ne \res(b'_{K,1,j'})$ for each $j \ne j'$,
so we can apply Corollary 3.5.2 of \cite{CGH} (which relates a function on $\FF_{q_K}$ to its Fourier transform) to
\[
\tilde{f}:\FF_{q_K} \to \CC : \tilde{y} \mapsto \sum_{j} c_{1,j} \psi_q(\res(b'_{K,1,j})\cdot \tilde{y})
\]
and get a $\tilde{y}_0 \in \FF_{q_K}$ with
$|c_{1,1}|_{\CC} \le |\tilde{f}(\tilde{y}_0)|_{\CC}$.
Set $B'_{\xi} := \{\xi + \gamma_{K,x,r} y' \mid y' \in\res^{-1}(\tilde{y}_0)\}$.
Since
$f_{1,K,\psi}(x,y) = \tilde{f}(\res(y'))$ and
$|f_{1,1,K,\psi}|_{\CC} = |c_{1,1}|_{\CC}$, we are done.
\end{proof}

\begin{proof}[Proof of Proposition \ref{IML} for $m>1$]
We proceed by induction on $m$.
Denote $(y_1,\ldots,y_{m-1})$ by $\hat y$.
Apply the $m=1$ case using $(x,\hat y)$ as parameters and
$y_m$ as the only $y$-variable.
This yields in particular an integer $d_1>0$, a surjection $\varphi_1:U
\to V_1$, 
and an expression of each $f_\ell$ as a sum of terms of the form
$G_{1}(\varphi_1(x,y))\psi(h_1(\varphi_1(x,y),y_m))$, where we omit the indices
$\ell, i$ to
simplify notation.

Now apply the induction hypothesis to the collection of functions
$G_{1}$,
this time using $\hat y$ as the $y$-variables, and the variables $(x,
r_1)$ as parameters running over $V_1$.
This yields an integer $d_2$, a surjection $\varphi_2:V_1\to V_2$
and an expression of each $G_1$ as a sum of terms of the form
$G_{2}(x,r)\psi(h_2(x,y,r))$, where $\varphi_2(\varphi_1(x,y))=(x,r)$.

Now define $\varphi$ as $\varphi_2\circ \varphi_1$ and $d=d_1 + d_2$.
Then 1) is satisfied and 2) also follows easily.
\end{proof}

\subsection{Expansions}\label{sec:expan}

All results, statements and definitions of Section \ref{sec:i} except Theorem \ref{badlocus} hold when one consequently replaces the meaning of definable by subanalytic, that is, one replaces $\gLPas$ by an enrichment obtained by adding some analytic structure as in \cite{CLip} to $\gLPas$. (Theorem \ref{badlocus} is about Zariski closed sets and becomes different and more technical in the subanalytic case, where one has to use systems of power series that can be interpreted as converging analytic functions on $\cO_K^n$ when $K$ in $\Locp$, see~\cite{CLip}.) Similarly, enriching $\gLPas$ by putting arbitrary additional structure on the residue ring sorts $\Res_n$ does not impair the results of Section \ref{sec:i}. Also, one can add constants for a ring of integers $\cO$ of a number field to $\gLPas$ in the sort $\VF$ and work uniformly in all finite field extensions of completions of the fraction field of $\cO$, see the appendix of \cite{CGH5} for details. The justification for these claims is that the results of Section \ref{sec:qe} can easily be adapted to such enrichments of $\gLPas$.

\section{Quantifier elimination and related results}
\label{sec:qe}

This section contains the key technical results. The novelty lies in the combination of removing quantifiers over the valued field and over the value group variables without restrictions on the ramification degree. The proof consists of replacing a quantifier over the value group $\VG$ by a quantifier bounded to some segment in $\VG$ and then replace such a bounded quantifier by a quantifier over a residue ring. The elimination of valued field quantifiers only is more classical and can be proved in the line of \cite{Pas2}, see e.g.~\cite{Basarab} and the variants in \cite{Rid} where this is done using model theoretic methods, and \cite{Flen} where this is done in the line of Cohen's method of \cite{Cohen}.

The main aspect in which the results here are different than the ones restricting to large residue field characteristic, is that one needs reparameterizations by the sorts $\RF_n$ (as in Definition \ref{param}) to get things working in the value group, while previously only finite partitions into definable parts were needed to exploit properties of the value group. Reparameterizations by the residue field were already needed in the more classical case from Pas \cite{Pas} on to understand subtle information about the valued field in terms of the residue field via cell decomposition. So, in some sense,
here we just need to reparameterize more often and into deeper residue rings. This is the reason why many results, as e.g.~Theorem \ref{p2pexpmot} above,  go through as before.

\subsection{Quantifier elimination and orthogonality}

Let us write $\LPres$ for the Presburger language which consists of the symbols $+$, $\le$, $0$, $1$, and for each $d=2,3,4,\dots$, a symbol $\equiv_d$ to denote the binary
relation $x\equiv y \bmod d$, which means
$$
(\exists z) dz = x-y,
$$
with $dz$ standing for $\sum_{i=1}^d z$.
The element $1$ is interpreted in an ordered abelian group as the minimal positive element if there is such an element, and, by convention in this paper, as $0$ otherwise.
Note that with this meaning, $\LPres$ is a definitional expansion of the language of ordered abelian groups, and it makes sense in any ordered abelian group.

Let us denote by $\gLPas'$ the (definitional) expansion of $\gLPas$ given by putting the language $\LPres$ on the value group (it suffices to add the symbols $1$ and $\equiv_d$ since the other symbols are already there), and, for each integers $n>0$ dividing $m>0$, relation symbols $A_n$ for subsets of $\Res_n$, and function symbols $\res_n: \VF\to \Res_n$,
 $\res_{m,n}:\Res_m\to \Res_n $, and $\cross_n : \VGinf \to \Res_n$.

An $\gLPas$-structure $L$ naturally extends to an $\gLPas'$-structure: the maps $\res_n$ and $\res_{m,n}$ are as in Section \ref{genval}, the set $A_{n,L}$ consists of the image under $\res_n$ of the elements in $\cO_L$ with $\ac_n(x)=1$, and, for any $n>0$, the map $\cross_n : \VGinf{}_L\to \RF_{n,L}$ sends $\gamma\in \VG_L$ to $\res_n(x)$ for any $x\in L$ with $\ac_n(x)=1$ and $\ord(x)=\gamma$ (in particular, $\cross_n(\gamma) = 0$ for $\gamma < 0$), and sends $+\infty$ to $0$.

Let us write $\gTPas'$ for the corresponding $\gLPas'$-theory.

The following result by S.~Rideau in \cite{Rid} (as a variation on results by Basarab in \cite{Basarab}) is obtained in loc.~cit.~from quantifier elimination in a closely related language (with so-called leading term structures or $\rv$-structure). Alternatively, one can note that the proofs of Pas \cite{Pas2} or of Flenner \cite{Flen} (both similar to the Cohen-Denef method \cite{Cohen}, \cite{Denef2}) can be adapted to yield direct proofs of the following quantifier elimination result.

\begin{thm}[\cite{Rid}]\label{QERid}
The theory $\gTPas$ eliminates valued field quantifiers in the language $\gLPas'$, even resplendently, relatively to the sorts $\VGinf$ and $\Res_n$ for $n>0$.
\end{thm}

For more context on `resplendent quantifier elimination' we refer to \cite{Rid} but let us recall that it means in Theorem \ref{QERid} that for any expansion $\cL$ of $\gLPas'$ which adds new language symbols only involving variables of the sorts $\RF_n$ and $\VGinf$, the expansion $\cL$ still eliminates valued field quantifiers.
Rideau \cite{Rid} uses, among other things, a slightly different language than $\gLPas'$, but with the same definable sets.

\par

We give an addendum to Theorem \ref{QERid} to eliminate also $\VGinf$-quantifiers for two kinds of value groups.

\par
Write $\TPres$ for the Presburger theory, namely, the $\LPres$-theory of $\ZZ$.
Write $\Tdoag$ for the theory of divisible ordered abelian  groups. By $\gTPas\cup \TPres$ we denote the theory $\gTPas$ together with $\TPres$ in the value group sort and likewise for the theories $\gTPas\cup \Tdoag$, $\gTPas'\cup \TPres$ and $\gTPas'\cup \Tdoag$.
\begin{thm}\label{QEZ}
The theories $\gTPas'\cup \TPres$ and $\gTPas'\cup \Tdoag$ eliminate quantifiers over the valued field and over the value group, in the language $\gLPas'$.
\end{thm}

\begin{remark}\label{rem:expansions}
Both the theories from Theorem \ref{QEZ} even resplendently eliminate valued field and value group quantifiers in $\gLPas'$, relatively to the sorts $\Res_n$, where resplendent relative to the $\Res_n$ means that new language symbols can be introduced only involving variables of the sorts $\RF_n$, $n>0$. An analytic structure from \cite{CLip} can also be joined to the language, with similar results.
\end{remark}

\begin{proof}[Proof of Theorem \ref{QEZ}]
We prove the result for $\gTPas'\cup \TPres$. The proof for $\gTPas'\cup \Tdoag$ is similar.

In the proof, we will slightly abuse notation when speaking about $\LPres$-formulas:
We will consider the constant terms of the form $\ord(n)$ for integers $n>0$
as $\LPres$-terms and they may as such appear in $\LPres$-formulas.


The first step is to use Theorem \ref{QERid} to reduce the problem to eliminating $\VGinf$-quantifiers from formulas having variables only of the value group sort and
of the residue ring sorts. Indeed, by that theorem (and by syntactical
considerations), every $\gLPas'$-formula is equivalent to a boolean combination
of formulas of the form
\[
\varphi((\ord p_i(x))_i, (\ac_{n_i}(p_i(x)))_i, z, \xi),
\]
where $x$ is a tuple of $\VF$-variables,
$p_i$ are polynomials with integer coefficients, $z$ is a tuple of $\VGinf$-variables, $\xi$ is a tuple of variables each of which runs over a residue ring, and $\varphi$ is a formula living only
in the value group and in the residue rings.

As usual, it suffices to eliminate a single existential quantifier, and we can assume that other value group quantifiers
have already been eliminated. Using some more syntactical arguments, it suffices to
eliminate $\exists y$ from formulas of the form
\begin{equation}\label{varphiy}
\exists y\, \big( \theta(y,z) \wedge \psi\left(\xi, \cross_{n_1}(t_1(y,z)),\ldots, \cross_{n_k}(t_k(y,z))\right) \big),
\end{equation}
for some $\LPres$-terms $t_i$, some $k\geq 0$, some $n_i\geq 0$, and where
$y$ is a value group variable, $z$ is a tuple of value group variables, $\xi$ a tuple of variables each of which runs over a residue ring,
$\theta$ is an $\LPres$-formula and $\psi$ is a formula on the residue ring sorts.

We may assume that $y$ appears non-trivially in each $t_i$
(since otherwise, we can replace $\cross_{n_k}(t_i(z))$ by a new $\RF_{n_k}$ variable and eliminate $\exists y$ from the resulting formula).

\medskip

The second step consists in reducing to the case where $\theta(y, z)$ implies that for each $i$, $t_i(y,z)$ lies in $[0, \ord(n_i)]$.
This is achieved by introducing (into $\theta(y, z)$) a case distinction,
for each $i$, on whether $t_i(y,z)$ lies in $[0, \ord(n_i)]$ or not.
We then treat each case separately, and whenever $t_i(y,z)$ does not lie in the interval,
$\cross_{n_i}(t_i(y,z))$ can be replaced by $0$.

In the case where all $t_i$ disappear in this way, we can then eliminate $\exists y$ from $\exists y\, \theta(y,z)$ using Presburger quantifier elimination which works even with our extra constant symbols for the values $\ord(n)$ for integers $n>0$. Indeed, quantifier elimination is preserved under adding constant symbols.
Otherwise, the new $\theta(y, z)$ in particular implies
\begin{equation}\label{eq.bound}
0 \le ay + \alpha(z) \le \ord(n_1),
\end{equation}
where $ay + \alpha(z) = t_1(y,z)$,
the integer $a$ is a non-zero (by our assumption that each $t_i$ does depend on $y$ non-trivially), and where $\alpha(z)$ is an $\LPres$-term.
In other words, our quantifier over $y$ now is bounded.

\medskip

Our next goal is to simplify the bound (\ref{eq.bound}) on $y$ to one of the form
\begin{equation}\label{eq.bound2}
0 \le y \le \ord(n)
\end{equation}
(i.e., we want $\theta$ to imply (\ref{eq.bound2}) for some $n$, possibly after some change of variables and other manipulations).
To this end, first, we may assume $a \ge 1$ in (\ref{eq.bound}) (by otherwise turning it around and adapting $\alpha$).
Now we replace $\cross_{n_i}(t_i(y,z))$ by $\cross_{n_i^a}(at_i(y,z))$ and modify $\psi$ to
reconstruct $\cross_{n_i}(t_i(y,z))$ from this (by taking the $a$-th root in
$A_{n_i^a}$ and then the image in $\RF_{n_i}$ under $\res_{n_i^a,n_i}$).

In this way, the $y$-coefficients in all $t_i$ become divisible by $a$. We now replace $ay$ by $y'$ in each $t_i$, we replace $\theta(y,z)$ by a formula equivalent to
\[
y' \equiv 0 \mod a \ \ \wedge\ \  \theta(\frac{y'}{a}, z),
\]
and we replace $\exists y$ by $\exists y'$. This modification replaces $a$ by $1$ in (\ref{eq.bound}). Finally,
we get rid of the $\alpha(z)$ in (\ref{eq.bound}) by replacing $y$ by $y - \alpha(z)$ everywhere, and adapting the $t_i$ and $\theta$ correspondingly.

\medskip

Fix $i$ and write $t_i(y,z)$ as $by + t'(z)$ where $b$ is a positive integer and $t'(z)$ an $\LPres$-term. The bounds $0 \le y \le \ord(n)$ and $0 \le t_i(y,z) \le \ord(n_i)$ implied by $\theta$
also imply bounds on $t'(z)$ of a similar kind. This means that for some suitable integers $m, m' \ge 1$,
we can focus on $\cross_m(t'(z) + \ord(m'))$  and on $\cross_n(y)$ instead of on $\cross_{n_i}(t_i(y,z))$, in the sense that there exists a function definable purely in the residue rings
sending $(\cross_n(y), \cross_m(t'(z) + \ord(m')))$ to $\cross_{n_i}(t_i(y,z))$, whenever $\theta(y, z)$ holds.

By applying this for all $i$, we can replace $\psi\left(\xi, (\cross_{n_i}(t_i(y,z)))_i\right)$
by a formula of the form
\begin{equation}\label{cross-n}
\psi'\left(\xi,\ \cross_n(y),\ \ \cross_{m_1}(t'_1(z)),\ldots, \cross_{m_k}(t'_k(z))\right),
\end{equation}
where $\psi'$ lives only in the residue rings, that is, $\psi'$ involves variables only running over $\RF_n$ for some $n>0$ and no variables running over $\VF$ neither $\VG$.

\medskip

Using Presburger quantifier elimination and Presburger cell decomposition, $\theta(y,z)$ can be (piecewise) written in the form
\begin{equation}\label{theta0}
\theta_0(z)\ \wedge\ \underbrace{\beta_1(z) \le cy \le \beta_2(z)\ \wedge\ y\equiv c\bmod \ell}_{(*)}
\end{equation}
for some quantifier free $\LPres$-formula $\theta_0$, some integers $c,\ell \ge 1$ and some $\LPres$-terms $\beta_1, \beta_2$.
Using that $\theta(y,z)$ implies $0 \le y \le \ord(n)$, we obtain that $\theta_0(z)$ implies $0 \le \beta_j(z) \le  c \ord(n) $ for $j=1,2$.
This means that $(*)$ can be incorporated into $\psi'$, after adding
$\cross_{n^c}(\beta_1(z)))$ and $\cross_{n^c}(\beta_2(z)))$ as input to $\psi'$.
Now the only place where $y$ appears in the entire formula is the
$\cross_n(y)$ in (\ref{cross-n}), so the quantifier over $\VGinf$ can be replaced by a quantifier in $\RF_n$, running over the image of $\cross_n$,
which is definable without $\VGinf$-quantifiers using the relation symbol $A_n$ from $\gLPas'$.
\end{proof}

From the above quantifier elimination, one obtains the following more precise description of formulas, which
can serve as a replacement for orthogonality of the residue field and the value group; the argument is
the same as in the usual Denef--Pas QE setting. From this, we then get, also in the usual way,
a strong from of stable embeddedness of some collections of sorts (Corollary~\ref{stab-emb}).


\begin{thm}
\label{weakortho}
Any $\gLPas'$-formula in free variables $x,\xi,z$ in the valued field, the residue rings, resp.\ the value group, is $\gTPas'\cup \TPres$-equivalent, to a finite disjunction of formulas of the form
\begin{equation}\label{varphiy2}
\Theta\big(z,(\ord p_i(x))_i \big) \wedge \Phi\big(\xi,(\ac_{n}(p_i(x)))_i,(\cross_{n}(t_j((\ord p_i(x))_i, z))  )_j\big) ,
\end{equation}
where $\Phi$ is a formula on the residue ring sorts, $\Theta$ a quantifier free $\LPres$-formula, $k,\ell\geq 0$, $n\geq 1$, the $t_j$ for $j=1,\ldots, k$ are $\LPres$-terms, and the $p_i$ for $i=1,\ldots, \ell$ are polynomials in $x$. The same statement also holds for $\gTPas'\cup \Tdoag$ instead of $\gTPas'\cup \TPres$.
\end{thm}

\begin{cor}\label{stab-emb}
Given a definable set $X\subset \VF^n \times Z$, where $Z$ is a product of residue ring sorts, there exists
a definable set $X' \subset Y' \times Z$, where $Y'$ is also a product of residue ring sorts, and a definable map $f\colon \VF^n \to Y'$ such
that for any $y \in \VF^n$, we have $X_y = X'_{f(y)}$.

The same is true if we allow both $Z$ and $Y'$ to be a product of residue ring sorts and copies of the value group, and all of this holds
both in $\gTPas'\cup \TPres$ and $\gTPas'\cup \Tdoag$.
\end{cor}

Note that in contrast to the boundedly ramified setting, we do not know whether also the value group by itself is stably embedded.

\begin{remark}
For both the theories, statements similar to Theorem~\ref{weakortho} and Corollary~\ref{stab-emb} hold in a resplendent form relatively to the sorts $\Res_n$, namely with an expansion of $\gLPas'$-formula which only enriches the sorts $\Res_n$.
\end{remark}

\subsection{Understanding the value group using reparameterization}\label{sec:presburgering}

From now on, and until the end of the paper, we work with an $\gLPas$-theory $\cT$ containing the theory $\gTPas$ introduced in Section \ref{ss:gendp}.
(Note that we do not assume $\cT$ to be complete. Also note that certain expansions of $\gLPas$ as explained in Remark \ref{rem:expansions} can also be used here and until the rest of the paper.) By a $\cT$-definable set associated to a $\gLPas$-formula $\varphi$ we mean (as is common) the information consisting of $\varphi(\cK)$
for every model $\cK$ of $\cT$.

We deduce various results about definable sets and maps in the value group; we start with some preliminary definitions.

\begin{defn}\label{param}
By a \emph{reparameterization} of a $\cT$-definable set $X$ is meant a $\cT$-definable bijection
$\sigma:X\to X_{\rm par} \subset \prod_{i=1}^k \Res_{n_i} \times X$ over $X$ onto a set, often denoted by $X_{\rm par}$, for some $n_i$ and some $k$.
For a $\cT$-definable function $f$ on $X$, we write $f_{\rm par}$ for the composition of $f$ with $\sigma^{-1}$.
\end{defn}

\begin{defn}
Given $\cT$-definable sets $Y$ and $X\subset Y\times \VG^m$, we call a map $f\colon X \to \VG$ \emph{linear
over $Y$} if there exist rational numbers $r_0, r_1, \dots, r_m$ and a map $\gamma\colon Y \to \VG$ such that
\[
f(y, z_1, \dots, z_m) = \gamma(y) + r_0 + r_1 z_1 + \dots + r_m z_m
\]
for every $(y, z_1, \dots, z_m) \in X$. (Note that if $f$ is $\cT$-definable, then $\gamma$ is $\cT$-definable.) We call a map  $f\colon X \to \VG^n$ \emph{linear
over $Y$} if each of its coordinate functions is linear over $Y$.
\end{defn}

In $\cT$, the structure on the value group is not necessarily the pure ordered abelian group structure. However,
the following corollaries give structural results about $\cT$-definable sets in the value group under $\TPres$ and $\Tdoag$.


Note in the following corollary that $\gamma \times \id_{\VG^m}(X_{\rm par})$ is not assumed to be equal to $X'$.

\begin{cor}\label{linear}
Suppose that $\cT$ contains either $\TPres$ or $\Tdoag$ on the value group.
Let $Y$ and $X\subset Y\times \VG^m$ be $\cT$-definable. Then there exist
\[
\begin{array}{ccccc}
Y \times \VG^m  & \xrightarrow{\ \ \sigma\ \ }  &   \Res_n^n\times Y\times \VG^m & \xrightarrow{\ \gamma \times \id_{\VG^m}\ } & (\VGinf)^k \times \VG^m\\
   \cup & & \cup & & \cup \\
 X & \hskip-5em\rlap{$\xrightarrow{\hskip5em}$} & X_{\rm par} & \hskip-8em\rlap{$\xrightarrow{\hskip8em}$} & X',
\end{array}
\]
where $\sigma$ is reparameterization, $X_{\rm par} = \sigma(X)$,
$\gamma\colon\Res_n^n\times Y \to (\VGinf)^k$ is a $\cT$-definable map
for some $k \ge 0$,
$X'$ is $\cL_{\rm oag}$-definable, and for every $z \in \Res_n^n\times Y$, we have $X_{{\rm par},z} = X'_{\gamma(z)}$.

If we additionally are given finitely many $\cT$-definable functions $f_1, \dots, f_\ell:X\to \VG$,
then we may moreover achieve that there exists
a finite $\cL_{\rm oag}$-definable partition of $X'$ such that for each part $A'$ and each $i$, the restriction of $f_{i,\rm par} = f_i \circ \sigma^{-1}$ to
$(\gamma \times \id_{\VG^m})^{-1}(A')$ is linear over $\Res_n^n\times Y$.
\end{cor}

\begin{proof}
By Theorem \ref{weakortho}, we may assume that $X$ is given by an $\gLPas'$-formula of the form (\ref{varphiy2}),
which in this context can be written as
\begin{equation}\label{form-linear}
\Theta(g(y), x) \wedge \Phi\big(h(y),\cross_{n}(t(g(y), x))\big) ,
\end{equation}
where $y$ runs over $Y$, $x$ runs over $\VG^m$, $\Theta$ an $\LPres$-formula, $\Phi$ is a formula on the residue ring sorts,
$t\colon \VG^{n+m} \to \VG^n$ is an $\cL_{\rm oag}$-definable function, and $g \colon Y \to \VG^n$ and $h \colon Y \to \RF_n^n$ are $\cT$-definable functions (for some $n$
which we may assume to be the same everywhere for simplicity). Here, $\cross_{n}$ is applied to a tuple by applying it to each coordinate individually.

We do a reparameterization $\sigma$ with new variables
\begin{equation}\label{repar-linear}
\zeta := \zeta(y,x) := \cross_{n}(t(g(y), x)).
\end{equation}
Then for $z = (\zeta, y) \in \RF_n^n \times Y$, we have
\[
X_{{\rm par}, z} = \begin{cases}
               \{x \mid \Theta(g(y), x) \wedge \cross_{n}(t(g(y), x)) = \zeta\} & \text{if $\Phi(h(y),\zeta)$ holds}\\
               \emptyset & \text{otherwise}.
              \end{cases}
\]

Now we define $\gamma\colon \RF_n^n \times Y \to (\VGinf)^{2n+1}$
\[
\gamma(\zeta, y) :=
\begin{cases}
(g(y),  \cross^{-1}_{n}(\zeta), \ord(n)) & \text{if } \Phi(h(y),\zeta) \text{ holds}\\
(\infty, \dots, \infty) & \text{otherwise.}
\end{cases}
\]
Here, if for some coordinate $\zeta_i$, $\cross_n^{-1}(\zeta_i)$ is not well-defined,
we use $\infty$ as preimage. Finally, we set
\begin{align*}
X' := \{&(u_1, u_2, u_3, x) \in \VGinf^n \times \VGinf^n \times \VGinf \times \VGinf^m \mid\\
&\Theta(u_1, x) \wedge \cross_n(t(u_1, x)) = \cross_n(u_2) \wedge u_3 \ne \infty
\}
\end{align*}
It is clear that $X_{{\rm par}, z} = X'_{\gamma(z)}$, and to define $X'$ entirely in $\cL_{\rm oag}$,
note that $\cross_n(a) = \cross_n(a')$ is expressible in $\cL_{\rm oag}$ using $u_3 = \ord(n)$.

To obtain the second part, we also define the graph of each $f_i$ by a formula as in (\ref{form-linear}), namely
\begin{equation}\label{form-linear2}
\Theta(g_i(y), x, x') \wedge \Phi\big(h_i(y),\cross_{n}(t_i(g_i(y), x, x'))\big) ,
\end{equation}
where $x$ still ranges over $\VG^m$ and $x'$ ranges over $\VG$.

This time, we reparameterize $X$ not only using (\ref{repar-linear}), but in addition using
\[
\zeta_i := \zeta_i(y,x) :=\cross_{n}(t_i(g_i(y), x, f_i(y, x))),
\]
and in the definition of $\gamma$, we insert additional coordinates
\[
\cross^{-1}_{n}(\zeta_1), \dots,  \cross^{-1}_{n}(\zeta_\ell).
\]
With these definitions, we obtain that $f_{i,\rm par} := f_i \circ \sigma^{-1}$ is equal to $f'_i \circ (\gamma \times \id_{\VG^m})$
for some $\cL_{\rm oag}$-definable $f'_i\colon X' \to \VG$. Using that $\cL_{\rm oag}$-definable functions are piecewise linear, choose the partition of $X'$ in such a way that
each $f'_i$ is linear over $(\VGinf)^k$ on each part $A'$.
\end{proof}

%

The following corollary describes definable functions from residue rings into the value group:
Piecewise, such functions range over an interval whose length is bounded by $\ord(n)$ for some integer $n \ge 1$.

\begin{cor}\label{RF-to-VG}
Suppose that $\cT$ contains either $\TPres$  or $\Tdoag$ on the value group.
For any $\cT$-definable function $f\colon Y \times \RF_m^m \to \VG$, where $Y$ is an arbitrary $\cT$-definable set,
there exists an integer $n \ge 1$ and finitely many $\cT$-definable functions $g_i\colon Y \to \VG$ such that
for every $(y, \xi) \in Y \times \RF_m^m$, one has
\begin{equation}\label{eq-RF-to-VG}
g_i(y) \le f(y, \xi) \le g_i(y) + \ord(n)
\end{equation}
for some $i$.
\end{cor}

\begin{proof}
Let the graph of $f$ be given by an $\gLPas'$-formula $\varphi(y,\xi,z)$, with $y\in Y$, $\xi\in \RF_m^m$ and $z\in \VG$. We assume that
$\varphi$ is of the form given by Theorem~\ref{weakortho}, which means in our context that it is a disjunction of formulas of the form
\begin{equation}\label{woapp}
\Theta\big(z, s(y) \big) \wedge \Phi\big(\xi,s'(y),(\cross_{n}(s_j(y) + b_j z) )_j\big) ,
\end{equation}
where $\Theta$ lives in the value group, $\Phi$ lives in the residue rings, the $b_j$ are integers, and
$s$, $s'$ and $s_j$ are $\cT$-definable functions from $Y$ to $\VG^N$, $\RF_M^M$ and to $\VG$, respectively.
That this defines the graph of a function implies that for each $y$ and $\xi$, there exists one clause of the form
(\ref{woapp}) which holds for exactly one $z$. If all $b_j$ are zero, or, if $s_j(y) + b_j z$ lies outside $[0,\ord n]$ when $z=f(y, \xi)$, then the Corollary follows from (\ref{woapp}), even with $n=1$. Otherwise, if $b_j \ne 0$ for some $j$ and $0 \le s_j(y) + b_j z \le \ord(n)$  when $z=f(y, \xi)$,
then an inequality of the form (\ref{eq-RF-to-VG}) follows as desired.
\end{proof}

By combining Corollaries \ref{linear} and \ref{RF-to-VG}, we obtain that even without reparameterization, definable functions
in the value group are piecewise ``approximatively linear'':

\begin{cor}\label{app-lin}
Suppose that $\cT$ contains either $\TPres$  or $\Tdoag$ on the value group.
Let $Y$ and $X\subset Y\times \VG^m$ be $\cT$-definable, and let
$$
f:X\to \VG
$$ be a $\cT$-definable function. Then there exists an integer $n \ge 1$, a finite partition of $X$ into parts $A$,
and for each part $A$ a map $g\colon X \to \VG$ which is linear over $Y$
such that $0 \le f(x) - g(x) \le \ord(n)$ for all $x \in A$.
\end{cor}

\begin{proof}
Apply Corollary \ref{linear} to $X$ and $f$
and then Corollary~\ref{RF-to-VG} to the function $f_{\rm par} = f \circ \sigma^{-1} \colon X_{\rm par} \to \VG$.
This yields an $n$ and finitely many functions $g_i\colon X \to \VG$ such that
for each $x \in X$, we have
\begin{equation}\label{ali}
g_i(x) \le f_{\rm par}(\sigma(x)) = f(x) \le g_i(x) + \ord(n)
\end{equation}
for some $i$. Partition $X$ according to the smallest $i$ for which (\ref{ali}) holds (for any ordering of the index set).
\end{proof}

The following result is specific to the Presburger group situation and goes back to the parametric rectilinearization result of \cite{CPres}.
\begin{prop}[Rectilinearization]\label{prop:recti}
Suppose that $\cT$ contains $\TPres$ on the value group.
Let $Y$ and $X\subset Y\times \VG^m$ be $\cT$-definable sets. Then there exist a finite partition of $X$ into
$\cT$-definable sets $A$ and for each part $A$ a reparameterization
$$
\sigma:A\to A_{\rm par}\subset \Res_n^n\times A,
$$
a set $B\subset \Res_n^n\times Y\times \VG^{m}$,
and a $\cT$-definable map $\rho:A_{\rm par}\to \VG^m$ which is linear over $\Res_n^n\times Y$ such that the following holds.

For each $y \in \Res_n^n\times Y$, the map $\rho(y, \cdot)$ is a bijection
from $A_{{\rm par},y}$ to $B_y$, and the set
$B_y$ is of the form
$\Lambda_y\times \NN^\ell$ for a bounded Presburger definable subset $\Lambda_y\subset \NN^{m-\ell}$.
Here, $\Lambda_y$ may depend on $y$, but the integer $\ell\geq 0$ only depends on $A$.
\end{prop}

In Proposition \ref{prop:recti}, by ``$\Lambda_y$ is Presburger definable'' we mean: There exists
a $\cT$-definable function $\gamma \colon Y \to \VG^k$ and a Presburger formula $\phi$ in $m-\ell+k$ free variables, such that for every $y$, $\phi(\VG, \gamma(y)) =\Lambda_y$.
Also, by being bounded for a  $\cT$-definable subset $S\subset \VG^n$ we mean that
$\exists r\in \VG : \forall s\in S\ \sum_{i=1}^n |s_i| < r$ holds.

\begin{proof}
The case where $Y$ lives in the value group and everything is Presburger-definable is Theorem~3 of
\cite{CPres}. (In that case, no reparameterization is necessary.)
Using Corollary~\ref{linear}, it is straight forward to reduce to that case:
We apply the corollary to $X$ and then (using the notation from that corollary) \cite[Theorem~3]{CPres} to
the resulting $X' \subset \VG^k \times \VG^m$. This yields a finite partition
of $X'$ into parts $A'$, and for each part $A'$,
a set $B' \subset \VG^k \times \VG^m$ and a $\cT$-definable map
$\rho:A' \to \VG^m$ such that for each $y' \in \VG^k$,
$\rho(y', \cdot)\colon A'_{y'} \to B'_{y'}$ is a bijection
and $B_{y'}$ is of the desired form.

Pulling $A'$, $B'$ and $\rho$ back via the map
\[
\gamma \times \operatorname{id}\colon \pi(X_{\rm par}) \times \VG^m \to \VG^k \times \VG^m
\]
yields a partition of $X_{\rm par}$ into pieces $A_{\rm par}$ with the desired properties,
and $\sigma^{-1}(A_{\rm par})$ yields the desired partition of $X$.
\end{proof}

\subsection{Cell decomposition and the Jacobian property}

Here we recall and adapt some terminology regarding cells and the Jacobian property. Theorem~\ref{cd} follows directly from results of \cite{CLip}, without using
the above new quantifier elimination.
Recall that $\cT$ is any $\gLPas$-theory containing $\gTPas$ (or, more generally, with $\gLPas$ replaced by a language according to Remark \ref{rem:expansions}.

Let $Y$ be a $\cT$-definable set.
The graph of a $\cT$-definable function $Y\to \VF$ is called a presented $0$-cell over $Y$.
A presented $1$-cell over $Y$ is
a $\cT$-definable set $X\subset Y\times \VF$ of the form
$$
\{(y,t)\mid y\in Y,\ t\in \VF, \ord (t - c(y) )\in G_y,\ \ac_n(t-c(y))=\xi(y)\}
$$
for some $\cT$-definable functions $c:Y\to \VF$ (called center), $\xi:Y\to \Res_n^\times$, a nonempty definable set $G\subset Y\times \VG$ and $G_y\subset \VG$ its fiber over $y\in Y$, and an integer $n>0$ (called depth).
Here, $\Res_n^\times$ denotes the group of units in the ring $\Res_n$.

The cell decomposition below says that, after reparameterization, every definable set is a finite union of presented cells.

Let $n\geq 0$ be an integer.
Say that a $\cT$-definable function $f:X\subset Y\times \VF\to \VF$ with $X$ a presented $1$-cell over $Y$ has the $n$-Jacobian property over $Y$ if, for each $y\in Y$, $f(y,\cdot)$ is injective on $X_y$ and for each ball $B$ contained in $X_y$, one has that $f(y,\cdot)$ has a derivative $f'(y,\cdot)$ of constant valuation and constant $\ac_n$ on $B$, $f(y,\cdot)$ maps $B$ onto a ball and, for all $x_1,x_2\in B$ one has
$$
\ord \big( f(y,x_1) - f(y,x_2) \big) = \ord  \big( f'(y,\cdot)_{|x_1} (x_1 - x_2)  \big)
$$
and
$$
\ac_n \big( f(y,x_1) - f(y,x_2) \big) = \ac_n  \big( f'(y,\cdot)_{|x_1} (x_1 - x_2)  \big).
$$

\begin{thm}[Cell decomposition and Jacobian property]\label{cd}
Let $Y$ and $X\subset Y\times \VF$ be $\cT$-definable, and let
$
f_i:X\to \VF 
$
be $\cT$-definable functions for $i=1,\ldots,N$. Moreover, let an integer $n \ge 1$ be given. Then there exist $m\geq 0$ and a finite partition of $X$ into parts $A$, and for each part a reparameterization
$$
\sigma:A\to A_{\rm par}\subset \Res_m^m\times X
$$
onto a presented cell $A_{\rm par}$ over $\Res_m^m\times Y$ such that each $f_{i, \rm par}$ either factorizes through the projection to $\Res_m^m\times Y$, or, has the $n$-Jacobian property over $\Res_m^m\times Y$.
\end{thm}

\begin{proof}
This follows from the resplendent forms of the corresponding results relative to $\RV_n$-sorts of Section 6 of \cite{CLip}, namely Theorem 6.3.7 and Remark 6.3.16. To translate between the terminology of \cite{CLip}, \cite{CLb} and of this section, one uses model theoretic compactness. The resplendency aspect of \cite{CLip} is used to put extra structure on $\RV_n\setminus\{0\}$ so that it becomes in a definable way bijective with $\Res_n^\times \times\VG$.
\end{proof}

\bibliographystyle{amsplain}
\bibliography{anbib}

\end{document}